\documentclass[11pt]{amsart}

\usepackage{amssymb}
\usepackage{mathrsfs}
\usepackage{multicol}

\usepackage[shortlabels]{enumitem}
\usepackage{cleveref}

\usepackage{graphicx}
\usepackage{tikz}
\usetikzlibrary{decorations.pathreplacing}
\usepackage{stmaryrd}

\makeatletter
\@namedef{subjclassname@2010}{%
  \textup{2010} Mathematics Subject Classification}
\makeatother

\newtheorem{thm}{Theorem}[section]
\newtheorem{corollary}[thm]{Corollary}
\newtheorem{lemma}[thm]{Lemma}
\newtheorem{proposition}[thm]{Proposition}

\theoremstyle{definition}
\newtheorem{defin}[thm]{Definition}
\newtheorem{remark}[thm]{Remark}

\numberwithin{equation}{section}

\newcommand{\mfrak}[1]{\mathfrak{#1}}
\newcommand{\mbb}[1]{\mathbb{#1}}
\newcommand{\mcal}[1]{\mathcal{#1}}

\newcommand{\bigslant}[2]{\left.\raisebox{.2em}{$#1$}\middle/\raisebox{-.2em}{$#2$}\right.}
\newcommand{\dinf}{d}
\newcommand{\resum}{\sideset{}{'} \sum}

\DeclareMathOperator{\Gal}{Gal}
\DeclareMathOperator{\Syl}{Syl}
\DeclareMathOperator{\Hom}{Hom}
\DeclareMathOperator{\rank}{rank}
\DeclareMathOperator{\GL}{GL}
\DeclareMathOperator{\cores}{cores}
\DeclareMathOperator{\ord}{ord}
\DeclareMathOperator{\Reg}{Reg}
\DeclareMathOperator{\id}{id}

\DeclareMathOperator{\im}{im}

\newcommand{\Cl}{\mathfrak{C}}

\begin{document}




\title[Ray Class Groups]{Gauss Sums, Stickelberger's Theorem and the Gras Conjecture for Ray Class Groups}

\author[T. All]{Timothy All}
\address{Department of Mathematics\\ Rose-Hulman Institute of Technology\\ 5500  Wabash Ave \\ Terre Haute, IN, U.S.A}
\email{timothy.all@rose-hulman.edu}


\date{}

\begin{abstract}
Let $k$ be a real abelian number field and $p$ an odd prime not dividing $[k:\mathbb{Q}]$. For a natural number $d$, let $E_d$ denote the group of units of $k$ congruent to $1$ modulo $d$, $C_d$ the subgroup of $d$-circular units of $E_d$, and $\mfrak{C}(d)$ the ray class group of modulus $d$. Let $\rho$ be an irreducible character of $G=\Gal(k/\mathbb{Q})$ over $\mathbb{Q}_p$ and $e_{\rho} \in \mathbb{Z}_p[G]$ the corresponding idempotent. We show that if the ramification index of $p$ in $k$ is less than $p-1$, then $|e_{\rho} \Syl_p(E_d/C_d) | = |e_{\rho} \Syl_p(\mfrak{C}_d)|$ where $\mfrak{C}_d$ is the part of $\mfrak{C}(d)$ where $G$ acts non-trivially. This is a ray class version of the Gras Conjecture. In the case when $p \mid [k:\mathbb{Q}]$, similar but slightly less precise results are obtained. In particular, beginning with what could be considered a Gauss sum for real fields, we construct explicit Galois annihilators of $\Syl_p(\mfrak{C}_{\mfrak{a}})$ akin to the classical Stickelberger Theorem.
\end{abstract}



\maketitle

\section{Introduction}
Let $k$ denote a real abelian number field with Galois group $G$. Let $\mfrak{o}_k=\mfrak{o}$ denote the ring of integers of $k$. Fix $d \in \mbb{N}$ and let $E_{\dinf}$ denote the units of $\mfrak{o}$ that are congruent to $1$ modulo $d$. Let $\overline{d}$ denote the product of distinct prime divisors of $d$, and for every $n \in \mbb{N}$ we let $\zeta_n$ stand for a primitive $n$-th root of unity. We assume the $\zeta_n$ have been chosen so that for every $t \mid n$ we have $\zeta_n^t = \zeta_{n/t}$. For every $n \in \mathbb{N}$, let $k^n = \mathbb{Q}(\zeta_n) \cap k$, and for $n >1$ satisfying $n \nmid \overline{d}$, let
\[ \delta_{n,d}:= N^{\mbb{Q}(\zeta_n)}_{k^n} \prod_{t \mid \overline{d}} (1-\zeta_n^t)^{\mu(t)d/t} \in k^{\times} \]
where $\mu(t)$ denotes the M\"obius function. Let $D(\dinf)$ denote the $G$-module generated by the $\delta_{n,d}$ for all $n \nmid \overline{d}$ in $k^{\times}$. We let
\begin{align*}
C(\dinf)&= E \cap D(\dinf) \\
C_{\dinf} &= E_{\dinf} \cap D(\dinf).
\end{align*}
We call the modules $D(\dinf)$, $C(\dinf)$, and $C_{\dinf}$ the $d$-cyclotomic numbers, units, and units congruent to $1$ modulo $d$, respectively. We write $D$ for $D(1)$ and $C$ for $C(1)$. Note that $C(1)=C_1$. These modules were originally introduced by Sinnott \cite{Sinnott} (for $d=1$) and by Schmidt \cite{Schmidt} (for $d>1$). Note that $C$ is not quite the full Sinnott group of cyclotomic units but rather its subgroup of totally positive units.

For an ideal $\mfrak{a} \subseteq \mfrak{o}$, let $\Cl(\mfrak{a})$ denote the ray class group of $k$ of modulus $\mfrak{a}$, and let $H(\mfrak{a})$ denote the corresponding ray class field of $k$ so that $\Gal(H(\mfrak{a})/k) \simeq \Cl(\mfrak{a})$ via the Artin map. Let $H_{\mfrak{a}} = H(\mfrak{a}) \cap \mbb{Q}^{ab}$, the maximal sub-extension of $H(\mfrak{a})/k$ abelian over $\mbb{Q}$, and let $\Cl_{\mfrak{a}} \leq \Cl(\mfrak{a})$ such that $\Cl_{\mfrak{a}} \simeq \Gal(H(\mfrak{a})/H_{\mfrak{a}})$. In the case when $p \nmid [k:\mbb{Q}]$, note that
\[ \Syl_p(\Cl_{\mfrak{a}}) = (1-e_1) \Syl_p\big( \Cl(\mfrak{a}) \big),\]
where $e_1 \in \mbb{Z}_p[G]$ is the idempotent associate to the trivial character.

There's a fascinating interplay between units structures and ideal structures in algebraic number theory. For example, the following theorem was proven by Sinnott \cite[Theorem 4.1]{Sinnott} for $d=1$, and by Schmidt \cite[Satz 3]{Schmidt} for $d>1$ using similar methods.
\begin{thm} \label{index}
If $p \nmid 2 \cdot |G|$, then $ \left| \Syl_p(E_{\dinf}/C_{\dinf}) \right| = \left| \Syl_p(\Cl_{\dinf}) \right|$.
\end{thm}
One of the aims of this article is to prove a Galois-equivariant version of the theorem above. To be precise, let $[\chi]$ be the $\Gal(\mbb{Q}_p(\zeta_{|G|})/\mbb{Q}_p)$-orbit of a non-trivial character $\chi$ of $G$ and define $\rho= \sum_{\psi \in [\chi]} \psi$. Let $\epsilon_{\rho}$ be the $\mbb{Z}_p$-valued idempotent
\[ \epsilon_{\rho} = \frac{1}{|G|} \sum_{\sigma \in G} \rho(\sigma) \sigma^{-1} \in \mbb{Z}_p[G].\]
For a $\mbb{Z}_p[G]$-module $M$, we let $M_{\rho}$ denote the sub-module $\epsilon_{\rho} M$.

Let $e(p)$ denote the ramification index of $p$ in $k$. One of our main results is the following
\begin{thm}\label{main}
If $p \nmid 2 \cdot |G|$ and $e(p) < p-1$, then
\[ \left| \Syl_p(E_d/C_d)_{\rho} \right| = \left| \Syl_p(\Cl_d)_{\rho} \right|.\]
\end{thm}
\Cref{main} is a ray class version of the Gras Conjecture \cite{Gras}, the statement of the claim when $d=1$. Greenberg \cite{Greenberg} observed that the Gras Conjecture followed from the Main Conjecture of Iwasawa Theory which was later on proven by Mazur and Wiles \cite{MazurWiles}.

In the case when $p$ possibly divides the order of $G$, we prove a result akin to Rubin's \cite[Theorem 1.3]{Rubin} which itself was a generalization of a theorem of Thaine \cite[Theorem 3]{Thaine}. Our method of proof follows along those same lines. In particular, we define a subgroup $\mathcal{S}(\mfrak{a})$ of $E$ which we call the \emph{$\mfrak{a}$-special units}. These are akin to Rubin's special units \cite{Rubin}, and we show that the $d$-cyclotomic units of Schmidt are a special instance of $d$-special units. We then show
\begin{thm} \label{rayrubin}
Let $\alpha: E \to \mcal{O}[G]$ be any $G$-module map where $\mcal{O}$ is the valuation ring of any finite extension of $\mbb{Q}_p$, and let $\varsigma_{\mfrak{a}} \in \mcal{S}(\mfrak{a})$. Then $\alpha(\varsigma_{\mfrak{a}})$ annihilates $\Cl_{\mfrak{a}} \otimes_{\mbb{Z}} \mcal{O}$.
\end{thm}

As stated, our method of proof originates in the work of Thaine and Rubin. Thaine noticed that cyclotomic units could be used to generate elements $\alpha$ that act like real analogues of Gauss sums much like roots of unity are used to generate classical Gauss sums. To generate $\alpha$, Thaine relied on an invocation of Hilbert's Theorem 90. A key feature here is that we give $\alpha$ explicitly. This affords finer control over the ideal relations revealed by the factorization of $\alpha$ thus paving the way towards annihilation results concerning ray classes. In particular we prove the following ray class version of a conjecture of D. Solomon \cite[Conjecture 4.1]{Solomon} which acts as a sort of Stickelberger Theorem for ray class groups.

\begin{thm} \label{raysolomon}
Let $\mcal{O}$ denote the valuation ring of a $p$-adic completion of $k$, and let $\varpi \in \mcal{O}$ be a local parameter. For every $\varsigma_{\mfrak{a}} \in \mcal{S}(\mfrak{a})$ we have that
 \[ \frac{|e(p)|_p^{-1}}{\varpi^{|e(p)|_p^{-1}}} \sum_{\sigma \in G} \log_p\big( \varsigma_{\mfrak{a}}^{\sigma} \big) \sigma^{-1} \in \mcal{O}[G] \]
annihilates $\Cl_{\mfrak{a}} \otimes_{\mbb{Z}} \mcal{O}$.
\end{thm}

\section{Preliminaries}

In this section, we collect some results that will be useful in the sequel concerning the structure of relevant $G$-modules contained in $k$. Until further notice, we consider $p$ to be an odd prime not dividing $[k:\mbb{Q}]$. Let
\begin{align*}
K &= \text{a local field containing the character values of $G$}\\
\mcal{O} &= \text{the valuation integers of $K$} \\
\mbb{F} &= \text{the residue field of $\mcal{O}$.} \\
\mbb{F}_p &=\text{the finite field with $p$-elements.}
\end{align*}
We normalize the $p$-adic absolute value in the usual way; $|p|_p = p^{-1}$. For any finite set $X$, we use $|X|$ to denote the number of elements in $X$. For $H \subseteq G$, we write $s(H)$ to denote the sum
\[ s(H) = \sum_{\sigma \in H} \sigma \in \mathbb{Z}[G].\]
We write $\widehat{G}$ for $\Hom_{\mbb{Z}}(G, \mcal{O}^{\times})$. For every $\chi \in \widehat{G}$, we let
\[ e_{\chi} = \frac{1}{|G|} \sum_{\sigma \in G} \chi(\sigma) \sigma^{-1} \in \mcal{O}[G], \]
the idempotent associate to $\chi$. We may naturally view $e_{\chi} \in \mbb{F}[G]$.

Throughout we use $\otimes$ as an abbreviation for $\otimes_{\mbb{Z}}$. For a $\mbb{Z}[G]$-module $M$ and commutative ring $R$, we make $M\otimes R$ into an $R[G]$-module in the obvious way. The following proposition generalizes \cite[Theorem 3.3]{All2}. It will be useful for demonstrating certain $G$-modules are cyclic later on.

\begin{proposition} \label{cyclicmodules}
Let $M$ be a free $\mbb{Z}$-submodule of either $(k,+)$ or $(k^{\times},\cdot)$ of finite rank such that $\sigma(M) =M$ for all $\sigma \in G$. For $H \leq G$, let $M^H$ denote the collection of all elements of $M$ fixed by $H$. For every $H \leq G$, suppose that the inclusion $M^H \subseteq M$ induces
\[ M^H \otimes \mbb{F}_p \hookrightarrow M \otimes \mbb{F}_p.\]
If there exists $m \in M$ such that $[M:\langle m \rangle_{\mbb{Z}[G]}]< \infty$, then $M \otimes \mbb{F}_p$ is a cyclic $\mbb{F}_p[G]$-module.
\end{proposition}
\begin{proof}
Suppose $k/\mbb{Q}$ is cyclic with $\sigma$ generating $G$. Let $r = \rank_{\mbb{Z}} M$ and let $\varrho: G \to \GL(r,\mbb{Z})$ be the representation induced by the action of $G$ on a fixed $\mbb{Z}$-basis, say $\{ m_1,m_2,\ldots, m_r\}$, for $M$. Let $m_{\varrho(\sigma)}$ and $h_{\varrho(\sigma)}$ be the minimal and characteristic polynomials for $\varrho(\sigma)$, respectively. Suppose $m \in M$ such that the index $[M: \langle m \rangle_{\mbb{Z}[G]}]$ is finite. From this it follows that $r= \rank_{\mbb{Z}} \langle m \rangle_{\mbb{Z}[G]}$ and
\[  h_{\varrho(\sigma)}(x) =m_{\varrho(\sigma)}(x) \mid x^{|G|}-1.\]
Now, let $\overline{\varrho} : G \to \GL(r,\mbb{F}_p)$ be the representation induced by the action of $G$ on the $\mbb{F}_p$-basis $\{ m_1 \bmod{pM}, \ldots, m_r \bmod{pM}\}$. Note that
\[ h_{\overline{\varrho}(\sigma)} \equiv h_{\varrho(\sigma)} \bmod{p}.\]
Since $p \nmid |G|$, it follows that $h_{\overline{\varrho}(\sigma)}$ factors into a product of distinct irreducibles modulo $p$. So $M/pM \simeq M \otimes \mbb{F}_p$ is a cyclic $\mbb{F}_p[G]$-module.

Now suppose $k$ is merely abelian over $\mbb{Q}$, and let $\chi \in \widehat{G}$. Then $\chi$ naturally determines a character of $G'=G/\ker \chi$, the Galois group of the cyclic extension $k'/\mathbb{Q}$ where $k'=k^{\ker \chi}$. We write $M'$ for $M^{\ker \chi}$, and let $e_{\chi}'$ denote the idempotent associate to $\chi \in \Hom_{\mbb{Z}}(G', \mcal{O}^{\times})$, i.e.,
\[ e_{\chi}' = \frac{1}{|G'|} \sum_{\sigma \in G'} \chi(\sigma) \sigma^{-1}.\]
Note that $e_{\chi} = |\ker \chi|^{-1} \cdot \cores_{k'}^k e_{\chi}'$ where $\cores_{k'}^k$ denotes co-restriction from $k'$ to $k$. We may naturally view $e_{\chi}$ and $e_{\chi}'$ as being $\mbb{F}$-valued, in particular, we may view $\widehat{G}$ as a $\Gal(\mbb{F}/\mbb{F}_p)$-module. Let $[\chi]=\{ \chi^{\tau}: \tau \in \Gal(\mbb{F}/\mbb{F}_p) \}$, and let $e_{[\chi]},e_{[\chi]}' \in \mbb{F}_p[G]$ be defined by
\[ e_{[\chi]} = \sum_{\psi \in [\chi]} e_{\psi} \qquad \text{and} \qquad e_{[\chi]}'=\sum_{\psi \in [\chi]} e_{\psi}'.\]
It follows that $e_{[\chi]} = |\ker \chi|^{-1} \cdot \cores_{k'}^k e_{[\chi]}'$.

By assumption, we have $M' \subseteq M$ induces $M' \otimes \mbb{F}_p \hookrightarrow M \otimes \mbb{F}_p$. So we may view $M' \otimes \mbb{F}_p \subseteq M \otimes \mbb{F}_p$. For $m \in M$, we have
\[ e_{[\chi]} (m \otimes 1) = \cores_{k'}^k e_{[\chi]}' \big( m \otimes (|\ker \chi|^{-1}) \big) = e_{[\chi]}' \big( m^{s(\ker \chi)} \otimes (|\ker \chi|^{-1}) \big).\]
Since $p \nmid n$, the map $M\otimes\mbb{F}_p \to M' \otimes \mbb{F}_p$ defined by $x \mapsto x^{s(\ker \chi)}$ is surjective. It follows that
\begin{equation} \label{equality} e_{[\chi]}(M \otimes \mbb{F}_p) = e_{[\chi]}' (M' \otimes \mbb{F}_p) \subseteq M \otimes \mbb{F}_p.\end{equation}
Now, viewing $M'$ as a $G'$-module, we note that $k'/\mbb{Q}$ is cyclic and $M'$ satisfies all the hypotheses of the proposition; $M'$ is a free $\mbb{Z}$-module of finite rank that is preserved under the action of $G'$, and for all $H' \subseteq G'$ we have $(M')^{H'} \subseteq M'$ induces $(M')^{H'} \otimes \mbb{F}_p \hookrightarrow M' \otimes \mbb{F}_p$ (otherwise there exists $H \leq G$ such that $M^H \otimes \mbb{F}_p \not \hookrightarrow M \otimes \mbb{F}_p$ contrary to assumption). So $M' \otimes \mbb{F}_p$ is a cyclic $\mbb{F}_p[G']$ module, whence $M' \otimes \mbb{F}_p$ is a cyclic $\mbb{F}_p[G]$-module.

Let $\mfrak{m}' \in M'\otimes \mbb{F}_p$ such that $\mfrak{m}'$ generates $M' \otimes \mbb{F}_p$ as an $\mbb{F}_p[G]$-module, and let $\mfrak{m}_{[\chi]}= e_{[\chi]} \mfrak{m}' \in M \otimes \mbb{F}_p$. Using \Cref{equality}, we get
\[ e_{[\chi]} (M \otimes \mbb{F}_p) = e_{[\chi]}'(M' \otimes \mbb{F}_p) = \langle \mfrak{m}_{[\chi]} \rangle_{\mbb{F}_p[G]}.\]
Let $\mfrak{m} \in M \otimes \mbb{F}$ be defined by $\mfrak{m} = \sum \mfrak{m}_{[\chi]}$
where the sum is taken over $X$, a complete system of representatives of $\widehat{G}/\Gal(\mbb{F}/\mbb{F}_p)$. Since $e_{[\chi]} \mfrak{m} = \mfrak{m}_{[\chi]}$, we have
\begin{align*}
M \otimes \mbb{F}_p &= \bigoplus_{\chi \in X} e_{[\chi]} (M \otimes \mbb{F}_p) \\
&= \bigoplus_{\chi \in X} \langle \mfrak{m}_{[\chi]} \rangle_{\mbb{F}_p[G]} \\
&= \langle \mfrak{m} \rangle_{\mbb{F}_p[G]}.
\end{align*}
This completes the proof of the proposition.
\end{proof}

\begin{lemma} \label{cycliclemma}
Let $M$ be a $\mbb{Z}[G]$-module. The following are equivalent:
\begin{enumerate}[\upshape (i)]
\item $M \otimes \mbb{Z}_p$ is a cyclic $\mbb{Z}_p[G]$-module.
\item $M \otimes \mbb{F}_p$ is a cyclic $\mbb{F}_p[G]$-module,
\end{enumerate}
\end{lemma}
\begin{proof}
For $\mfrak{m} \in M\otimes \mbb{Z}_p$, we have the exact sequence
\[ \langle \mfrak{m} \rangle_{\mbb{Z}_p[G]} \otimes \mbb{F}_p \to (M \otimes \mbb{Z}_p) \otimes \mbb{F}_p \simeq M \otimes \mbb{F}_p \to \left( \bigslant{M\otimes \mbb{Z}_p}{\langle \mfrak{m} \rangle_{\mbb{Z}_p[G]}} \right) \otimes \mbb{F}_p \to 0.\]
If $\langle \mfrak{m} \rangle_{\mbb{Z}_p[G]} = M \otimes \mbb{Z}_p$, then the third term in this exact sequence is zero. Thus the first map must be onto, so $\overline{\mfrak{m}}$, the image of $\mfrak{m}$ through the first map, generates $M \otimes \mbb{F}_p$. So (i) implies (ii).

Conversely, suppose $\overline{\mfrak{m}} \in M \otimes \mbb{F}_p$ such that $\langle \overline{\mfrak{m}} \rangle_{\mbb{F}_p[G]} = M \otimes \mbb{F}_p$. Let $\mfrak{m} \in M \otimes \mbb{Z}_p$ reduce to $\overline{\mfrak{m}}$ in $M \otimes \mbb{F}_p$. Then the first map in the exact sequence above is onto, hence the third term in that sequence is zero. It follows that $\langle \mfrak{m} \rangle_{\mbb{Z}_p[G]} = M \otimes \mbb{Z}_p$. So (ii) implies (i).
\end{proof}

\begin{corollary} \label{specificcyclic}
The modules $E \otimes \mbb{Z}_p$ and $\mfrak{o} \otimes \mbb{Z}_p$ are cyclic $\mbb{Z}_p[G]$-modules. In fact
\[ E \otimes \mbb{Z}_p \simeq \bigslant{\mbb{Z}_p[G]}{\big( s(G) \big)} \quad \text{and} \quad \mfrak{o} \otimes \mbb{Z}_p \simeq \mbb{Z}_p[G].\]
\end{corollary}
\begin{proof}
Let $H \leq G$, then $E^H$ is the set of units of $k^H$ and $\mfrak{o}^H$ is the ring of integers of $k^H$. Since $k$ is real and Galois, it follows that
\[ E^H \otimes \mbb{F}_p  \hookrightarrow E \otimes \mbb{F}_p.\]
Similarly, we have that
\[ \mfrak{o}^H \otimes \mbb{F}_p = \mfrak{o}^H/p \hookrightarrow \mfrak{o}/p= \mfrak{o} \otimes \mbb{F}_p.\]
The cyclicality of $E \otimes \mbb{Z}_p$ and $\mfrak{o} \otimes \mbb{Z}_p$ now follows \Cref{cyclicmodules} and \Cref{cycliclemma}. The particular isomorphisms given in the claim are now straightforward to prove.
\end{proof}

For a prime $\ell$, we adopt the following notation to be used throughout.
\begin{align*}
\sigma_{\ell}&= \text{a Frobenius automorphism in $G$ for $\ell$,} \\
I_{\ell}&= \text{the inertia subgroup in $G$ of $\ell$,} \\
e_{\ell} &= \frac{s(I_{\ell})}{|I_{\ell}|}.
\end{align*}
We may consider $e_{\ell} \in \mathbb{Z}_p[G]$ since $p \nmid [k:\mathbb{Q}]$.

\begin{corollary} \label{rays}
Suppose $\ell \neq p$ is a rational prime and $\mfrak{L}$ the product of primes of $\mfrak{o}$ over $\ell$. Then for a positive integer $e$, we have
\[ \left( \bigslant{\mfrak{o}}{\mfrak{L}^e} \right)^{\times} \otimes \mbb{Z}_p \simeq
\bigslant{\mbb{Z}_p[G]}{ \big(\ell - \sigma_{\ell} \cdot e_{\ell} \big)}. \]
If $e(p)$ is less than $p-1$, then
\[ \left( \bigslant{\mfrak{o}}{p^e} \right)^{\times} \otimes \mbb{Z}_p \simeq \bigslant{\mbb{Z}_p[G]}{\big( p^{e-1}(p- \sigma_p \cdot e_p) \big)}.\]
\end{corollary}

\begin{proof}
Suppose $\ell \neq p$. Then
\[ \left( \bigslant{\mfrak{o}}{\mfrak{L}^e} \right)^{\times} \otimes \mbb{Z}_p \simeq \left( \bigslant{\mfrak{o}}{\mfrak{L}} \right)^{\times} \otimes \mbb{Z}_p \simeq \prod_{\mfrak{l} \mid \ell} \left( \bigslant{\mfrak{o}}{\mfrak{l}} \right)^{\times} \otimes \mbb{Z}_p,\]
where the product is over all prime ideals $\mfrak{l}$ of $\mfrak{o}$ dividing $\ell$. Fix such an ideal $\mfrak{l}$ and let $u \in \mfrak{o}^{\times}$ such that $u \bmod{\mfrak{l}}$ is a generator for the group $(\mfrak{o}/\mfrak{l})^{\times}$ and $u \equiv 1 \bmod{\mfrak{l}^{\sigma}}$ for all $\sigma \in G$, $\sigma \neq 1$. We have
\[ \left( \bigslant{\mfrak{o}}{\mfrak{L}} \right)^{\times} \otimes \mbb{Z}_p = \langle \mfrak{u} \rangle_{\mbb{Z}_p[G]}, \qquad \text{where} \quad  \mfrak{u}=u \otimes 1 \in \left( \bigslant{\mfrak{o}}{\mfrak{L}} \right)^{\times} \otimes \mbb{Z}_p. \]
Consider the surjective map
\begin{align*}
\varphi: \mbb{Z}_p[G] &\to \left( \bigslant{\mfrak{o}}{\mfrak{L}} \right)^{\times} \otimes \mbb{Z}_p \\
\theta &\mapsto \mfrak{u}^{\theta}.
\end{align*}
Clearly, $\big( \ell- \sigma_{\ell} \cdot e_{\ell} \big) \subseteq \ker \varphi$. The claim will now follow from the fact that the quotient ring $\mbb{Z}_p[G]$ modulo the ideal $\big( \ell-\sigma_{\ell} \cdot e_{\ell} \big)$ has the correct order. To compute this order, we note that $\chi$ induces a homomorphism from $\mbb{Z}_p[G] \to \mcal{O}$ in a natural way so that
\[ \left| \bigslant{\mbb{Z}_p[G]}{(\ell- \sigma_{\ell} \cdot e_{\ell})} \right| = \prod_{\chi} \left| \chi( \ell - \sigma_{\ell} \cdot e_{\ell} ) \right|_p^{-1}.\]
Note that
\[ \chi \big( \ell-\sigma_{\ell} \cdot e_{\ell} \big) = \begin{cases}
\ell, & I_{\ell} \not \subset \ker \chi \\
\big( \ell - \chi(\sigma_{\ell}) \big), & I_{\ell} \subset \ker \chi.
\end{cases} \]
So
\[ \left| \bigslant{\mbb{Z}_p[G]}{\big( \ell-\sigma_{\ell} \cdot e_{\ell} \big)} \right| = \prod_{\chi \in \widehat{G'}} \left| \ell - \chi(\sigma_{\ell}) \right|_p^{-1},\]
where $G' = G/I_{\ell}$. Let $G'_{\ell} = G_{\ell}/I_{\ell}$ where $G_{\ell}$ is the decomposition group for $\ell$. Since the order of $\sigma_{\ell} \in G'_{\ell}$ is $f=[\mfrak{o}/\mfrak{l}: \mathbb{Z}/(\ell)]$, it follows that
\begin{align*}
\prod_{\chi \in \widehat{G'}} \big( \ell - \chi(\sigma_{\ell}) \big) &= \prod_{a=0}^{f-1} \big( \ell- \zeta_f^a \big)^r , \qquad r=[G:G_{\ell}] \\
&= \big( \ell^f -1 \big)^r.
\end{align*}
So
\[ \left| \bigslant{\mbb{Z}_p[G]}{\big( \ell-\sigma_{\ell} \cdot e_{\ell} \big)} \right| = |\ell^f - 1|_p^{-r} = \left| \left(\bigslant{  \mfrak{o}}{\mfrak{L}} \right)^{\times} \otimes \mbb{Z}_p \right| = \left| \left(\bigslant{  \mfrak{o}}{\mfrak{L}^e} \right)^{\times} \otimes \mbb{Z}_p \right| .\]
This proves the claim in the $\ell \neq p$ case.

Now, suppose $\ell = p$ and $e(p) < p-1$. For each prime $\mfrak{p}$ of $\mfrak{o}$ dividing $p$, let $U_{\mfrak{p}}^{(j)} = \{ x \in \mfrak{o}: x \equiv 1 \bmod{\mfrak{p}^j} \}$. Then we have
\begin{equation} \label{decomp} \left( \bigslant{\mfrak{o}}{p^e} \right)^{\times} \otimes \mbb{Z}_p \simeq \prod_{\mfrak{p} \mid p} \left( \bigslant{\mfrak{o}}{\mfrak{p}^{e \cdot e(p)}} \right)^{\times} \otimes \mbb{Z}_p \simeq \prod_{\mfrak{p} \mid p} \bigslant{U_{\mfrak{p}}^{(1)} }{U_{\mfrak{p}}^{(e \cdot e(p))}}. \end{equation}
For all $x \in \mfrak{o}$ satisfying $|x|_p < p^{-1/(p-1)}$, we have that $\log_p(1+x)$ and $\exp_p(x)$ are defined by their power series', moreover, $\exp_p \left( \log_p(1+x) \right) = 1+x$. Since $e(p) < p-1$, we have that $|1+x|_p< p^{-1/(p-1)}$ for all $1+x \in U_{\mfrak{p}}^{(1)}$. Thus the following map is a $\mbb{Z}_p[G_p]$-isomorphism:
\begin{align*}
\bigslant{U_{\mfrak{p}}^{(1)}}{U_{\mfrak{p}}^{(e \cdot e(p))}} &\to \bigslant{\mfrak{p}}{\mfrak{p}^{e\cdot e(p)}} \\
(1+x) &\mapsto \log_p(1+x) \bmod{\mfrak{p}^{e \cdot e(p)}}.
\end{align*}
From \Cref{decomp}, we now have
\[ \left( \bigslant{\mfrak{o}}{p^e} \right)^{\times} \otimes \mbb{Z}_p \simeq \prod_{\mfrak{p}\mid p} \bigslant{\mfrak{p}}{\mfrak{p}^{e\cdot e(p)}} \simeq \bigslant{\mfrak{P}}{p^e} \]
as $\mbb{Z}_p[G]$-modules where $\mfrak{P}$ is the product of primes of $\mfrak{o}$ over $p$. It follows from \Cref{cyclicmodules} that $\mfrak{P} \otimes \mbb{Z}_p$ is a cyclic module, so the above isomorphism tells us that $( \mfrak{o}/p^e )^{\times} \otimes \mbb{Z}_p$ is cyclic. It remains to show that $\mfrak{P}/(p^e)$ is isomorphic to $\mathbb{Z}_p[G]/(p^e-p^{e-1} \sigma_p \cdot e_p)$ as a $\mathbb{Z}_p[G]$-module.

Since $\mfrak{o}/(p^e)$ is cyclic, there is an onto homomorphism $\psi : \mathbb{Z}_p[G] \to \mfrak{o}/(p^e)$. Note that
\[ (p,1-e_p)= p \mbb{Z}_p[G] + (1-e_p) \mbb{Z}_p[G] \xrightarrow{\psi} \mfrak{P}/(p^e).\]
Moreover, since $e_{\chi} (1-e_p) = 0$ for all $\chi$ that are trivial on $I_p$ (otherwise $e_{\chi}(1-e_p) = e_{\chi}$), we have that
\[ \left| \bigslant{\mathbb{Z}_p[G]}{(p,1-e_p)} \right| = \prod_{\ker \chi \supseteq I_p} p = p^{[G:I_p]} = \left| \bigslant{\mfrak{o}}{\mfrak{P}} \right|, \]
so $\psi\big( (p,1-e_p) \big) = \mfrak{P}/(p^e)$. From the preceding paragraph, we know that $\mfrak{P}/(p^e)$ is also cyclic, so there exists an onto homomorphism $\psi': \mbb{Z}_p[G] \to \mfrak{P}/(p^e)$. Since $(p^e-p^{e-1} \sigma_p \cdot e_p) \cdot (p,1-e_p) \subseteq (p^e)$, it follows that $(p^e-p^{e-1} \sigma_p \cdot e_p) \subseteq \ker \psi'$.

\begin{multicols}{2}
\tikzset{>=latex}
\begin{tikzpicture}[x=1.9cm,y=1.5cm]
\draw node (Zp) at (0,0) {$\mathbb{Z}_p[G]$};
\draw node (1) at (0,-1) {$(p,1-e_p)$};
\draw node (2) at (0,-2) {$(p^e)$};
\draw node (a) at (1,0) {$\mfrak{o}$};
\draw node (b) at (1,-1) {$\mfrak{P}$};
\draw node (c) at (1,-2) {$(p^e)$};
\draw node (zp) at (2,-1) {$\mathbb{Z}_p[G]$};
\draw node (i) at (2,-2) {$\ker \psi'$};
\draw node (ii) at (2,-3) {$(p^e-p^{e-1}e_p)$};

\draw[->] (Zp) -- (a) node[midway,above] {$\psi$};
\draw[->] (1) -- (b);
\draw[->] (2) -- (c);
\draw (Zp) -- (1) -- (2);
\draw (a) -- (b) -- (c);
\draw (zp) -- (i)  -- (ii);

\draw[->] (zp) -- (b) node[midway,above] {$\psi'$};
\draw[->] (i) -- (c);
\end{tikzpicture}

Since $e_{\chi} e_p = 0$ for all $\chi$ that are non-trivial on $I_p$ (otherwise $e_{\chi} e_p =e_{\chi}$), we have that $\left| \bigslant{\mbb{Z}_p[G]}{(p^e-p^{e-1} \sigma_p \cdot e_p)} \right|$ equals
\begin{align*}
\prod_{\ker \chi \supseteq I_p} p^{e-1} \prod_{\ker \chi \not \supseteq I_p} p^e &= p^{|G|e - [G:I_p]} \\
&= \left| \bigslant{\mfrak{P}}{(p^e)} \right|.
\end{align*}
So $\ker \psi' = \big( p^e - p^{e-1} \sigma_p \cdot e_p \big)$, and the claim follows.
\end{multicols}
\end{proof}

\begin{remark} If one takes $E = M$ (or $\mfrak{o} = M$) in \Cref{cycliclemma}, then (i) and (ii) are also equivalent to the statement that there exists $\epsilon \in E$ (or $\alpha \in \mfrak{o}$) such that $[E: \langle \epsilon \rangle_{\mbb{Z}[G]}]$ (or $[\mfrak{o} : \langle \alpha \rangle_{\mbb{Z}[G]}]$) is finite and co-prime to $p$. In particular, there exists $\epsilon \in E$ such that $\epsilon \otimes 1$ generates $E \otimes \mbb{Z}_p$ as a $\mbb{Z}_p[G]$-module.

The results of \Cref{rays,specificcyclic} also hold under extension of scalars. For example, we have
\[ E \otimes \mcal{O} \simeq \bigslant{\mcal{O}[G]}{\big( s(G) \big)}, \quad \left( \bigslant{\mfrak{o}}{\mfrak{L}^e} \right)^{\times} \otimes \mcal{O} \simeq \bigslant{\mcal{O}[G]}{(\ell - \sigma_{\ell} \cdot e_{\ell})}.\]

\end{remark}

\section{The Ray Class Gras Conjecture}

In this section, we aim to prove \Cref{main}. \Cref{index} was proven by way of the mapping $l : k^{\times} \to \mathbb{R}[G]$ defined by
\[ l(x) = -\frac{1}{2} \sum_{\sigma \in G} \log| x^{\sigma}| \sigma^{-1}.\]
Since $E_d/C_d \simeq l(E_d)/l(C_d)$, the index $[E_d:C_d]$ may be studied by decomposing the index $[l(E_d):l(C_d)]$ into various parts. Our method will be similar, but we work $p$-adically.

\subsection{Notations \& Preliminaries}

Fix a positive integer $d$ and let $m$ be the conductor of $k$. Let $\Omega_p$ (resp. $\Omega$) denote the algebraic closure of $\mathbb{Q}_p$ (resp. $\mathbb{Q}$), and fix an embedding of $\Omega \hookrightarrow \Omega_p$ so that we may view $\Omega \subseteq \Omega_p$. Let $k_p$ denote the topological closure of $k$, and let
\begin{itemize}
\item $K$ denote the field $k_p$ adjoined with all character values of $G$,
\item $\mcal{O}$ denotes the valuation integers of $K$.
\end{itemize}
Throughout this section we assume that
\begin{itemize}
\item $n > 1$ and $n \nmid \overline{d}$ unless stated otherwise,
\item $p$ is an odd prime not dividing $[k:\mathbb{Q}]$,
\item $e(p)$ is the ramification index of $p$ in $k$ (or in $K$), and is less than $p-1$,
\item $f(p)$ is the residue of $p$ in $K$.
\end{itemize}

A \emph{lattice} $L$ in a $K$-vector space $V$ is a free $\mcal{O}$-module whose $\mcal{O}$-rank equals the $K$-dim of $KL$ and $V=KL$. If $L$ and $M$ are lattices in $V$ such that there exists an automorphism $\phi$ of $V$ where $\phi(L) = M$, then we define the generalized index $(L:M)$ by
\[ (L:M) = p^{f(p) \cdot \ord_{\varpi} \det \phi} = \left| \det \phi \right|_p^{-e(p) f(p)},\]
where $(\varpi)$ is the prime ideal of $\mcal{O}$. This index is independent of the choice of $\phi$, and if $M \subseteq L$ with $[L:M]<\infty$, then $(L:M) = [L:M]$.

If $M$ is a lattice contained in $K[G]$ and $\alpha \in K[G]$, we let $\alpha M$ denote lattice
\[ \alpha M = \{ \alpha m : m \in M\}.\]
In particular, we write $M_{\chi}$ for $e_{\chi} M$. If $M$ is additionally a $G$-module, we have the natural decomposition
\[ M = \bigoplus M_{\chi},\]
the direct sum taken over all $\chi$ such that $M_{\chi} \neq 0$. A character $\chi \in \widehat{G}$ naturally induces a ring homomorphism $K[G] \to K$. The index $(M : \alpha M)$ exists if and only if $\chi(\alpha) \neq 0$ whenever $M_{\chi} \neq 0$. If this is the case, then
\[ (M : \alpha M) = \prod \left| \chi(\alpha) \right|_p^{-e(p)f(p)},\]
the product taken over all $\chi$ such that $M_{\chi} \neq 0$.

For $\alpha \in \mcal{O}[G]$, note that $\chi(\alpha) e_{\chi} = e_{\chi} \alpha$. In this case, we define $\alpha^{-1} \in K[G]$ to be the element defined by
\[ \alpha^{-1} := \sum \chi(\alpha)^{-1} e_{\chi},\]
where the sum is taken over all characters $\chi$ satisfying $\chi (\alpha) \neq 0$. This gives us
\[ \alpha \cdot \alpha^{-1} = \sum e_{\chi},\]
the sum over all $\chi$ such that $\chi(\alpha) \neq 0$.

Since we're assuming $e(p) < p-1$, we get that $\log_p(k^{\times}) \subseteq \mcal{O}$ where $\log_p$ is the Iwasawa logarithm. Let $\vartheta: k^{\times} \to \mcal{O}[G]$ be the $G$-module map defined by
\[ \vartheta(x) = \sum_{\sigma \in G} \log_p(x^{\sigma}) \sigma^{-1}.\]
This map takes the place of the map $l$ defined at the forefront of this section. If $x \in \ker \vartheta$, then $x = \zeta p^r$ where $\zeta$ is a root of unity and $r$ is a rational number. Since $k$ is real and Galois, it follows that $\ker \vartheta$ consists of $\pm 1$ and either all powers of $p$ or all powers of $\sqrt{p}$. We fix the following notation to be used throughout:
\begin{align*}
\mathscr{E}_d &= \text{the ideal of $\mcal{O}[G]$ generated by $\vartheta(E_{\dinf})$} \\
\mathscr{C}_d &= \text{the ideal of $\mcal{O}[G]$ generated by $\vartheta(C_{\dinf})$}
\end{align*}
We omit subscripts when $d=1$. Since $\mcal{O}$ is a flat $\mbb{Z}$-module, we have that
\[ \bigslant{E_{\dinf} \otimes \mcal{O}}{C_{\dinf} \otimes \mcal{O}} \simeq \bigslant{E_d}{C_d} \otimes \mcal{O}.\]
Consider the natural map from $E_d \otimes \mcal{O} \to \mathscr{E}_d$ defined by $(\epsilon \otimes \alpha) \mapsto \vartheta(\epsilon) \alpha$ extended by linearity. Since $\ker \vartheta$ consists of $\pm 1\cdot$(rational powers of $p$) and $k$ satisfies Leopoldt's conjecture with $p$ an odd prime, it follows that
\[ E_d \otimes \mcal{O} \simeq \mathscr{E}_d.\]
In fact, we have
\[ \bigslant{E_d}{C_d} \otimes \mcal{O} \simeq \bigslant{\mathscr{E}_d}{\mathscr{C}_d}.\]
As in the previous section, if $H \subseteq G$, then we write $s(H)$ to denote the sum in $\mathbb{Z}[G]$ of those automorphisms in $H$. We write
\[ \resum_{\sigma \in G/H} \]
for the restricted sum over a system of unique representatives in $G$ of $G/H$. We also keep the notations for $\sigma_{\ell}$, $e_{\ell}$ and $I_{\ell}$ from the previous section.

We will often make use of the following proposition.

\begin{proposition} The group ring $\mcal{O}[G]$ is a principle ring.
\end{proposition}
\begin{proof}
Note that $\mcal{O}/(\varpi^n)$ is a local Artinian principle ring. Since $G$ is assumed to have no $p$-part, it follows that $\mcal{O}/(\varpi^n)[G] \simeq \mcal{O}[G]/(\varpi^n \mcal{O}[G])$ is a principal ring \cite[Theorem 4]{Dorsey}. So for an ideal $I$ of $\mcal{O}[G]$, for every $n \in \mathbb{N}$, there exists $\alpha_n \in \mcal{O}[G]$ such that
\[ \bigslant{I + \varpi^n\mcal{O}[G]}{\varpi^n \mcal{O}[G]} = \bigslant{(\alpha_n) + \varpi^n \mcal{O}[G]}{\varpi^n \mcal{O}[G]}.\]
Since $\mcal{O}[G]$ is compact (in the product topology) we have that $\alpha_n \to \alpha \in \mcal{O}[G]$, i.e., there exists $\alpha \in \mcal{O}[G]$ such that $\alpha \equiv \alpha_n \bmod{\varpi^n \mcal{O}[G]}$ for all $n \in \mathbb{N}$. We now show that $I=(\alpha)$. For each $\gamma \in I$, $n \in \mathbb{N}$, there exists $\beta,\beta_n \in \mcal{O}[G]$ such that $\alpha \beta_n \equiv \gamma \bmod{\varpi^n \mcal{O}[G]}$ and $\beta_n \to \beta$. Then $\gamma = \alpha \beta$, hence $I$ is principle generated by $\alpha$.
\end{proof}

\subsection{The Modules $\mathscr{E}$ and $\mathscr{C}$}

For a non-trivial character $\chi \in \widehat{G}$ of conductor $f_{\chi}$, let $L_p'(0,\chi)$ denote the special value
\[ L_p'(0,\chi) = \sum_{a=1}^{f_{\chi}} \log_p(1-\zeta_{f_{\chi}}^a) \overline{\chi}(a) \in \mcal{O}.\]
And let
\[ \omega' = \sum_{\chi \neq 1} L_p'(0,\chi) e_{\chi} \in \mcal{O}[G].\]

In the pages following, our goal will be to dissect the $\chi$-components of $\bigslant{\mathscr{E}_d}{\mathscr{C}_d}$. The following useful proposition is a $p$-adic formulation of a result of Sinnott.

 \begin{proposition} \label{padicsinnott}
Let $n > 1$, and write $k^n$ for $\mbb{Q}(\zeta_{n}) \cap k$, $G_n$ for $\Gal(k/k^{n})$, and $\delta_n^{(t)}$ for $N^{\mbb{Q}(\zeta_{n})}_{k^{n}}(1-\zeta_n^t)$. Then
\begin{equation} \label{sinnott} (1-e_1) \cdot \vartheta (\delta_n^{(t)}) = \omega' \cdot [\mathbb{Q}(\zeta_n): k^n \mathbb{Q}(\zeta_{n/t})] \cdot s\big( G_{n/t} \big) \cdot \prod_{\ell \mid n} \left( 1-\sigma_{\ell}^{-1} \cdot e_{\ell} \right). \end{equation}
\end{proposition}
\begin{proof}
Let $L$ and $R$ denote the respective left and right hand sides of \Cref{sinnott}. It suffices to show that $\psi(R) = \psi(L)$ for every $\psi \in \widehat{G}$.

Since
\[ \vartheta \big( \delta_n^{(t)} \big) = [\mathbb{Q}(\zeta_n):k^n \mathbb{Q}(\zeta_{n/t})] \cdot s(G_{n/t}) \cdot \resum_{\sigma \in G/G_{n/t}} \log_p \big(
\big( \delta_{n/t}^{(1)} \big)^{\sigma} \big) \sigma^{-1},\]
it follows that $\psi(R) = \psi(L)=0$ if $\psi$ is either the trivial character or is a non-trivial character on $G_{n/t}$.

Now, suppose $\psi\neq 1$ and $\psi$ is trivial on $G_{n/t}$. Note that
\[ \psi(R) = L_p'(0,\psi) [\mathbb{Q}(\zeta_n): k^n \mathbb{Q}(\zeta_{n/t})] \cdot |G_{n/t}| \cdot \prod_{\ell \mid n/t} \left( 1-\overline{\psi}(\ell) \right)\]
and
\[ \psi \big( \vartheta(\delta_n^{(t)}) \big) = [\mathbb{Q}(\zeta_n):k^n \mathbb{Q}(\zeta_{n/t})] \cdot |G_{n/t}| \cdot \resum_{\sigma \in G/G_{n/t}} \log_p\big( \big( \delta_{n/t}^{(1)} \big)^{\sigma} \big) \overline{\psi}(\sigma),\]
so
\[ \psi(L) = [\mathbb{Q}(\zeta_n): k^n \mathbb{Q}(\zeta_{n/t})] \cdot |G_{n/t}| \cdot \sum_{\substack{a=1 \\ (a,n/t)=1}}^{n/t} \log_p(1-\zeta_{n/t}^a)\overline{\psi}(a).\]
Write $f_{\psi}$ for the conductor of $\psi$. Since $G_{n/t} \leq \ker \psi$, it follows that $f_{\psi} \mid n/t$. We set $f_{\psi} g = n/t$ and interpret $\log_p(0) \cdot 0$ to be equal to $0$ to obtain
\begin{align*}
\sum_{\substack{a=1\\ (a,n/t)=1}}^{f_{\psi}g} \log_p(1-\zeta_{n/t}^a) \overline{\psi}(a) &= \sum_{a=1}^{f_{\psi} g} \log_p(1-\zeta_{n/t}^a) \overline{\psi}(a) \cdot \prod_{\ell \mid n/t} (1- \overline{\psi}(\ell)) \\
&= \sum_{a=1}^{f_{\psi}} \log_p(1-\zeta_{f_{\psi}}^a) \overline{\psi}(a) \cdot \prod_{\ell \mid n/t} (1- \overline{\psi}(\ell)) \\
&= L_p'(0,\psi) \cdot \prod_{\ell \mid n/t} (1- \overline{\psi}(\ell)).
\end{align*}
So $\psi(L) = \psi(R)$, and this completes the proof of the proposition.
\end{proof}

Using \Cref{specificcyclic}, we may fix an $\epsilon \in E$ such that $\vartheta(\epsilon)=\varepsilon$ generates $\mathscr{E}$ as an ideal in $\mathcal{O}[G]$. We define a special element $\omega \in \mathscr{C}$ by
\[ \omega = \sum_{\chi \neq 1} e_{\chi} \vartheta(\delta_{\chi}), \qquad \delta_{\chi} = N^{\mathbb{Q}(\zeta_{f_{\chi}})}_{k^{f_{\chi}}} (1-\zeta_{f_{\chi}}),\]
where $f_{\chi}$ is the conductor of $\chi$. This element generates $\mathscr{C}$:

\begin{corollary} \label{padicsinnottcor} The ideal $\mathscr{C} \subseteq \mcal{O}[G]$ is generated by $\omega$, in fact, we have
\[ \mathscr{C} = \omega \varepsilon^{-1} \cdot \mathscr{E}.\]
In particular, if $\chi \neq1$, then we have
\[ [\ \mathscr{E}_{\chi}: \mathscr{C}_{\chi} ] = \left| \chi(\varepsilon)^{-1} L_p'(0,\chi) \right|_p^{-e(p)f(p)}.\]
\end{corollary}
\begin{proof}

Let $\mathscr{D}$ denote the $\mcal{O}[G]$-ideal generated by $\vartheta(D)$. Note that $\mathscr{C}$ is the kernel in $\mathscr{D}$ of multiplication by $s(G)$, i.e.,
\[ (1-e_1) \mathscr{D} = \mathscr{C}.\]
Since $\omega = (1-e_1) \omega$, we apply \Cref{padicsinnott} to obtain
\[ \chi(\omega) = |G_{f_{\chi}}|  L_p'(0,\chi).\]
Since $|G_{f_{\chi}}|$ is a $p$-adic unit, it follows that
\[\chi(\varepsilon)^{-1} L_p'(0,\chi) \cdot e_{\chi} \in \left( \varepsilon^{-1} \mathscr{C} \right)_{\chi}.\]
Since
\[ \bigslant{\mathscr{E}}{\mathscr{C}} \simeq \bigslant{ \varepsilon^{-1} \mathscr{E} }{\varepsilon^{-1} \mathscr{C}} = \bigslant{(1-e_1) \mcal{O}[G]}{\varepsilon^{-1} \mathscr{C}}, \]
it follows that
\begin{equation} \label{chiparts} [\mathscr{E}_{\chi} : \mathscr{C}_{\chi}] \leq \left| \chi(\varepsilon)^{-1} L_p'(0,\chi) \right|_p^{-e(p)f(p)}.\end{equation}
Moreover, \Cref{index} gives us that
\[ \left| \Cl \otimes \mcal{O} \right| =  [\mathscr{E}:\mathscr{C}] \leq \prod_{\chi \neq 1} \left| \chi(\varepsilon)^{-1} L_p'(0,\chi) \right|_p^{-e(p)f(p)}.\]

Now, note that
\[ \prod_{\chi \neq 1} \chi(\varepsilon) = \det (\alpha \mapsto \varepsilon \alpha) = \Reg'_p \]
where $\Reg'_p$ differs from $\Reg_p$, the Leopoldt regulator of $k$, by a unit of $\mbb{Z}_p$. Substituting $\Reg'_p$ into the $p$-adic class number formula, we get
\begin{equation} \label{classnumber} |\Cl| =_p \frac{1}{\Reg_p} \prod_{\chi \neq 1} L_p'(0,\chi) =_p \prod_{\chi \neq 1} \chi(\varepsilon)^{-1} L_p'(0,\chi), \end{equation}
where $a=_p b$ means that $a$ and $b$ differ by a $p$-adic unit. So \Cref{classnumber} and \Cref{index} give
\begin{equation} \label{chiparts2} [\mathscr{E}: \mathscr{C}] = \prod_{\chi \neq 1} [\mathscr{E}_{\chi}: \mathscr{C}_{\chi}] \leq \prod_{\chi \neq 1}  \left| \chi(\varepsilon)^{-1} L_p'(0,\chi) \right|_p^{-e(p)f(p)} = [\mathscr{E}:\mathscr{C}] .\end{equation}
Putting \Cref{chiparts,chiparts2} together, it must be that
\[  [\mathscr{E}_{\chi}: \mathscr{C}_{\chi}] = \left| \chi(\varepsilon)^{-1} L_p'(0,\chi) \right|_p^{-e(p)f(p)}, \]
for all $\chi \neq 1$. It follows that
\[ \mathscr{C} = \omega \varepsilon^{-1} \cdot \mathscr{E} = \omega (1-e_1) \cdot \mathcal{O}[G] = \omega \cdot \mathcal{O}[G].\]
This completes the proof of the corollary.
\end{proof}

\subsection{The Modules $\mathscr{E}_d$ and $\mathscr{C}_d$}
We will make use of the following auxiliary $\mcal{O}[G]$-modules:
\begin{align*}
\mathscr{D}(d) &= \text{the ideal of $\mcal{O}[G]$ generated by $\vartheta(D(d))$} \\
\mathscr{C}(d) &= \text{the ideal of $\mcal{O}[G]$ generated by $\vartheta(C(d))$}.
\end{align*}
Again, we omit parentheses when $d=1$. For any divisor $t \mid n$, let $\alpha_n(t) \in \mcal{O}[G]$ be defined by
\[ \alpha_n(t) := [\mathbb{Q}(\zeta_n) : k^n \mathbb{Q}(\zeta_{n/t})] \cdot s(G_{n/t}) \cdot \prod_{\ell \mid n/t} (1- \sigma_{\ell}^{-1} \cdot e_{\ell}),\]
the product taken over primes $\ell$ dividing $\dfrac{n}{t}$. Let $\mathscr{U}$ denote the $\mcal{O}[G]$-ideal generated by these elements $\alpha_n(t)$ (for all $n \geq 1$, and all $t \mid n$). Then \Cref{padicsinnott} reads
\[ (1-e_1) \vartheta(\delta_n^{(t)}) = \omega' \cdot \alpha_{n}(t).\]
More generally, since $\omega' \cdot \alpha_1(1) = \omega' \cdot \alpha_n(n) = 0$, it follows that
\[ \mathscr{C} = (1-e_1) \mathscr{D} = \omega' \cdot \mathscr{U}.\]

We need to know what the above formula looks like if we replace $\mathscr{C}$ with $\mathscr{C}(d)$. Towards that end, for an integer $t>1$, we define
\[ \kappa_t = \prod_{\ell \mid t} (\ell - \sigma_{\ell}), \]
the product taken over all primes dividing $t$. We first make a few observations that whittle down the ideal $\mathscr{D}(d)$.

\begin{lemma} Let $\bar{d}_m = (\bar{d},m)$. Then
\[ \mathscr{D}(d) = \frac{d}{\bar{d}} \cdot \kappa_{\bar{d}/\bar{d}_m} \cdot \mathscr{D}(\bar{d}_m).\]
\end{lemma}
\begin{proof}
This follows from \cite[Lemmas 3.2 and 3.5]{Schmidt}.
\end{proof}

For $n \geq 1$, we let
\[ \alpha_{n,\bar{d}_m} := \kappa_{\bar{d}_m/(\bar{d}_m,n)} \cdot \sum_{t \mid (\bar{d}_m,n)} \mu(t) \cdot \frac{ (\bar{d}_m,n)}{t} \cdot \alpha_n(t).\]
Note that since $p \nmid [k:\mathbb{Q}]$ we get that $\alpha_{n,\bar{d}_m} \in \mcal{O}[G]$ for all $n \geq 1$.

\begin{corollary} \label{Uprime} If $n >1$ and $n \nmid \bar{d}_m$, then
\[ (1-e_1) \vartheta(\delta_{n,\bar{d}_m}) = \omega' \cdot \alpha_{n,\bar{d}_m}.\]
\end{corollary}
\begin{proof}
For a prime divisor $\ell$ of $\bar{d}_m$ such that $\ell \nmid n$, we have
\[ \delta_{n,\bar{d}_m} = N^{\mathbb{Q}(\zeta_n)}_{k^n} \prod_{t \mid \frac{\bar{d}_m}{\ell}}  \big( 1-\zeta_n^t \big)^{\mu(t) \bar{d}_m/(t\ell) \cdot \ell} \cdot \big(1- \zeta_n^t \big)^{\mu(t\ell) \bar{d}_m/(t\ell) \cdot \sigma_{\ell}} = \delta_{n,\bar{d}_m/\ell}^{\ell-\sigma_{\ell}}.\]
So
\[ \vartheta(\delta_{n,\bar{d}_m}) = \kappa_{\bar{d}_m/(\bar{d}_m,n)} \cdot \vartheta( \delta_{n,(\bar{d}_m,n)} ).\]
Since
\[ \vartheta( \delta_{n,(\bar{d}_m,n)} ) = \sum_{t \mid (\bar{d}_m,n)} \mu(t) \cdot \frac{(\bar{d}_m,n)}{t} \vartheta(\delta_n^{(t)}),\]
the corollary follows from \Cref{padicsinnott}.
\end{proof}

\begin{lemma} \label{restrictn}
If $\delta_{n,\bar{d}_m} \in D(\bar{d}_m)$ and $q$ is a prime dividing $(n,\bar{d}_m)$ such that $v_q(n) > v_q(m)$, then $\delta_{n,\bar{d}_m} =1$.
\end{lemma}
\begin{proof}
Note that
\[ \delta_{n,\bar{d}_m} = \prod_{t \mid \frac{\bar{d}_m}{q}} N^{\mbb{Q}(\zeta_n)}_{k^n} \left( \frac{ (1-\zeta_n^t)^q }{1-\zeta_n^{tq}} \right)^{\mu(t) \cdot \frac{\bar{d}_m/q}{t}}.\]
Since $v_q(n) > v_q(m)$, it follows that $k^n \subseteq \mathbb{Q}(\zeta_{n/q}) \subseteq \mathbb{Q}(\zeta_n)$. Now consider
\[ N^{\mbb{Q}(\zeta_n)}_{k^n} \left( \frac{ (1-\zeta_n^t)^q }{1-\zeta_n^{tq}} \right) = N^{\mathbb{Q}(\zeta_{n/q})}_{k^n} N^{\mathbb{Q}(\zeta_n)}_{\mathbb{Q}(\zeta_{n/q})} \left( \frac{ (1-\zeta_n^t)^q }{1-\zeta_n^{tq}} \right).\]
Since $\zeta_n^t = \zeta_{n/t}$ and $q^2 \mid n$, we have
\[ N^{\mathbb{Q}(\zeta_n)}_{\mathbb{Q}(\zeta_{n/q})} (1-\zeta_{n/t}) = 1-\zeta_{n/(tq)} ,\]
whence the claim of the proposition.
\end{proof}

\begin{proposition} \label{Cd} Let $\mathscr{U}(\bar{d}_m)$ be the $\mcal{O}[G]$-ideal generated by the $\alpha_{n,\bar{d}_m}$ with $n \mid m$. Then
\[ \mathscr{C}(d) = (1-e_1) \mathscr{D}(d) = \dfrac{d}{\bar{d}} \cdot \kappa_{\bar{d}/\bar{d}_m} \cdot \omega' \cdot \mathscr{U}(\bar{d}_m).\]
\end{proposition}
\begin{proof}
Let $\mathscr{U}'(\bar{d}_m)$ be the $\mcal{O}[G]$-ideal generated by $\alpha_{n,\bar{d}_m}$ satisfying $n > 1$ and $n \nmid \bar{d}_m$. The statement of the proposition then follows from \Cref{Uprime} if we replace $\mathscr{U}(\bar{d}_m)$ with $\mathscr{U}'(\bar{d}_m)$. So the proposition rests on showing that
\[ \omega' \cdot \mathscr{U}(\bar{d}_m) = \omega' \cdot \mathscr{U}'(\bar{d}_m).\]
Toward that end, note that from \Cref{restrictn}, it follows that
\begin{equation} \label{reduction} \alpha_{n,\bar{d}_m} = \alpha_{(n,m),\bar{d}_m} \cdot \prod_{\ell \mid \frac{\bar{n} }{(\bar{n},m)}} (1-\sigma_{\ell}^{-1}).\end{equation}
So $\mathscr{U}'(\bar{d}_m)$ is generated by those $\alpha_{n,\bar{d}_m}$ satisfying $n>1$, $n \nmid \bar{d}_m$, and $n \mid m$. Hence $\omega' \cdot \mathscr{U}'(\bar{d}_m) \subseteq \omega' \cdot \mathscr{U}(\bar{d}_m)$.

Going the other way, note that $\alpha_{1,\bar{d}_m} = \kappa_{\bar{d}_m} \cdot \alpha_1(1)$. Since $\alpha_1(1) = s(G)$, it follows that $\omega' \cdot \alpha_{1,\bar{d}_m} = 0$. Now suppose $\alpha_{n,\bar{d}_m} \in \mathscr{U}(\bar{d}_m)$ where $n \mid \bar{d}_m$. Let $q$ be a prime such that $q \nmid m$ and consider $\alpha_{nq,\bar{d}_m}$. From \Cref{reduction}, we have
\[ \alpha_{nq,\bar{d}_m} = \alpha_{n,\bar{d}_m} \cdot (1-\sigma_q^{-1}).\]
Since any given element of $G$ is the Frobenius of infinitely many primes, it follows that there exists a collection of primes $Q$ such that $|Q| = |G|$ and
\[ \sum_{q \in Q} \alpha_{nq,\bar{d}_m} = \alpha_{n,\bar{d}_m} \cdot \sum_{\tau \in G} (1-\tau^{-1}) = \alpha_{n,\bar{d}_m} \cdot (|G| - s(G)).\]
Hence
\[ \omega' \cdot \sum_{q \in Q} \alpha_{nq,\bar{d}_m} = \omega' \cdot \alpha_{n,\bar{d}_m} \cdot |G|.\]
Since $|G|$ is a $p$-adic unit, it follows that $\omega' \cdot \mathscr{U}(\bar{d}_m) \subseteq \omega' \cdot \mathscr{U}'(\bar{d}_m)$. This completes the proof of the proposition.
\end{proof}

It remains to determine a generator for $\mathscr{U}(\bar{d}_m)$. For $t \in \mbb{N}$, let $\varkappa_t$ be the element of $\mcal{O}[G]$ defined by
\[ \varkappa_{t} = \prod_{\ell \mid t} \big( \ell - \sigma_{\ell} \cdot e_{\ell} \big),\]
the product taken all primes $\ell$ dividing $t$.

\begin{proposition} \label{Udm} The ideal $\mathscr{U}(\bar{d}_m)$ of $\mcal{O}[G]$ is generated by $\varkappa_{\bar{d}_m}$.
\end{proposition}
\begin{proof}
Suppose $\chi=1$. If $n \mid \bar{d}_m$, then
\begin{align*}
\chi(\alpha_{n,\bar{d}_m}) &= \pm \varphi \left( \frac{\bar{d}_m}{(\bar{d}_m,n)} \right) \cdot [\mathbb{Q}(\zeta_n): k^n] \cdot |G| \\
&= \pm \varphi(\bar{d}_m) \cdot [k:k^n].
\end{align*}
Otherwise $\chi( \alpha_{n,\bar{d}_m}) = 0$. On the other hand, we have $\chi( \varkappa_{\bar{d}_m} ) = \varphi(\bar{d}_m)$. Since $[k:k^n]$ is a $p$-adic unit, it follows that $e_1 \mathscr{U}(\bar{d}_m) = e_1 \varkappa_{\bar{d}_m} \mcal{O}[G]$.

Now, let $\chi$ be a non-trivial character of $G$ with conductor $f$. Suppose $n \mid m$ such that $\chi$ is non-trivial on $G_{n/t}$ for all $t \mid (\bar{d}_m,n)$. Then
\[ \chi\big( s(G_{n/t}) \big) = 0,\]
so $\chi( \alpha_{n,\bar{d}_m}) = 0$. So in order to get a non-trivial contribution to the $\chi$-part of $\mathscr{U}(\bar{d}_m)$, we consider those $\alpha_{fN,\bar{d}_m}$ satisfying $N \geq 1$ and $fN \mid m$.

We have
\begin{equation} \label{Ud} \chi( \alpha_{fN,\bar{d}_m}) = \prod_{\ell \mid \frac{\bar{d}_m}{(\bar{d}_m, fN)}} (\ell - \chi(\sigma_{\ell})) \left[ \sum_{t \mid (\bar{d}_m,fN)} \mu(t) \cdot \frac{(\bar{d}_m,fN)}{t} \cdot \chi\big( \alpha_{fN}(t) \big) \right] ,\end{equation}
where
\[ \chi\big( \alpha_{fN}(t)\big) = \begin{cases}
    0 & f \nmid fN/t \\
    [\mathbb{Q}(\zeta_{fN}) : k^{fN} \mathbb{Q}(\zeta_{fN/t})] \cdot | G_{fN/t} | \cdot \displaystyle{\prod_{\ell \mid \frac{Nf}{t}}} \big( 1-\chi(\sigma_{\ell}^{-1} e_{\ell})\big) & f \mid fN/t.
    \end{cases} \]
So as far as $\chi( \alpha_{fN,\bar{d}_m})$ is concerned, we might as well only sum over all those $t \mid (\bar{d}_m, fN)$ satisfying $f \mid fN/t$. Such $t$'s must divide $N$, moreover, it's easy to see that
\[ [\mathbb{Q}(\zeta_{fN}): k^{fN} \mathbb{Q}(\zeta_{fN/t})] \cdot |G_{fN/t}| = [\mathbb{Q}(\zeta_{fN}): \mathbb{Q}(\zeta_{fN/t})] \cdot [k: k^{fN}].\]
Let $A_{fN,\bar{d}_m}$ denote the bracketed term in \Cref{Ud}. So far we have
\[ A_{fN,\bar{d}_m} = \sum_{t \mid (\bar{d}_m,N)} \mu(t) \cdot \frac{ (\bar{d}_m, fN) }{t} \cdot \frac{\varphi(fN)}{\varphi(fN/t)} \cdot [k:k^{fN}] \cdot \prod_{\ell \mid fN/t} \big( 1-\chi(\sigma_{\ell}^{-1}e_{\ell}) \big).\]

Write $N = N_1 N_2^2$ where $N_1$ is a product of distinct primes. Write $(\bar{d}_m, N) = Q_1 Q_2$ where $Q_1$ is the product of those primes that divide $N_1$ and not $N_2$. Then
$A_{fN,\bar{d}_m}$ equals
\[ \sum_{t_1 \mid Q_1} \mu(t_1) \cdot \frac{(\bar{d}_m,fN)}{t_1} \cdot [k:k^{fN}] \left[ \sum_{t_2 \mid Q_2} \frac{\mu(t_2)}{t_2} \cdot \frac{\varphi(fN)}{\varphi(fN/(t_1t_2))} \cdot \prod_{\ell \mid \frac{fN}{t_1 t_2}} \big( 1-\chi(\sigma_{\ell}^{-1}e_{\ell}) \big) \right].\]
Notice that if $Q_2 \neq 1$, then the bracketed term immediately above equals
\[ \varphi(t_1) \cdot \prod_{\ell \mid fN/t_1} \big( 1-\chi(\sigma_{\ell}^{-1}e_{\ell}) \big) \cdot \sum_{t_2 \mid Q_2} \mu(t_2) = 0.\]
So we must have that $N=N_1$, i.e., that $N$ is square-free otherwise $\chi( \alpha_{fN,\bar{d}_m}) = 0$. If we now were to go back to the beginning of this paragraph and let $Q_2$ denote those primes that divide both $N$ and $f$, then we could similarly deduce that $\chi( \alpha_{fN,\bar{d}_m})=0$ if $Q_2 \neq 1$ since the same bracketed term equals zero. Hence, $N$ must not only be square-free but also co-prime to $f$ otherwise $\chi( \alpha_{fN,\bar{d}_m}) = 0$.

Now, suppose $N$ is square-free and co-prime to $f$. Since $\chi(\sigma_{\ell}^{-1} e_{\ell}) = 0$ for all $\ell \mid f$, we now have
\[ A_{fN,\bar{d}_m} = (\bar{d}_m,f) \cdot [k:k^{fN}] \left[ \sum_{t \mid (\bar{d}_m,N)} \mu(t) \cdot \frac{(\bar{d}_m,N)}{t} \cdot \varphi(t) \cdot \prod_{\ell \mid N/t} \big( 1-\chi(\sigma_{\ell}^{-1} e_{\ell}) \big)\right] .\]
Let $B_{fN,\bar{d}_m}$ be the bracketed term in the formula immediately above. Let $q$ be a prime divisor of $(\bar{d}_m,N)$. Then
\[ B_{fN,\bar{d}_m} = \sum_{t \mid \frac{(\bar{d}_m,N)}{q}} \mu(t)  \frac{(\bar{d}_m,N)}{tq}  \varphi(t)  \prod_{\ell \mid N/(qt)} \big( 1-\chi(\sigma_{\ell}^{-1}e_{\ell}) \big) \cdot \big( q(1 - \chi(\sigma_q^{-1}e_q))  - \varphi(q) \big).\]
Repetition of the above on the remaining prime divisors of $(\bar{d}_m, N)$ yields
\[ B_{fN,\bar{d}_m} = \prod_{\ell \mid N/(\bar{d}_m,N)} \big( 1-\chi(\sigma_{\ell}^{-1}e_{\ell}) \big) \cdot \prod_{q \mid (\bar{d}_m,N)} \big( 1- q\chi(\sigma_q^{-1}e_q) \big).\]
This quantity is largest ($p$-adically) when $N \mid \bar{d}_m$. Since $N$ is assumed co-prime to $f$, we have that $\chi(e_q) = 1$ for all $q \mid N$, and for those $q \mid f$, we have $\chi(e_q)=0$. It follows that
\begin{align*}
A_{fN,\bar{d}_m} &= \pm \chi(N) \cdot (\bar{d}_m,f) \cdot [k:k^{fN}] \cdot \prod_{q \mid N} \big( q- \chi(\sigma_q e_q) \big) \\
&= \pm \chi(N) \cdot [k:k^{fN}] \cdot \prod_{q \mid (\bar{d}_m,fN)} (q- \chi(\sigma_qe_q)),
\end{align*}
where
\[ \chi(N) = \prod_{q \mid N} \chi(\sigma_q^{-1}).\]

Picking up where we left off with \Cref{Ud}, we've shown that $\chi( \alpha_{fN,\bar{d}_m})$ is largest ($p$-adically) when $N$ is square-free and co-prime to $f$ with $N \mid \bar{d}_m$ in which case we have
\[ \chi( \alpha_{fN,\bar{d}_m}) = \pm \chi(N) \cdot [k:k^{fN}] \cdot \prod_{\ell \mid \bar{d}_m} \big( \ell - \chi(\sigma_{\ell} e_{\ell}) \big).\]
Since $\chi( \alpha_{fN,\bar{d}_m})$ differs from $\chi( \varkappa_{\bar{d}_m})$ by a $p$-adic unit (and since $\chi$ was arbitrary), it follows that $e_{\chi} \mathscr{U}(\bar{d}_m) = e_{\chi} \varkappa_{\bar{d}_m} \mcal{O}[G]$.

We've shown that $e_{\chi} \mathscr{U}(\bar{d}_m) = e_{\chi} \varkappa_{\bar{d}_m} \mcal{O}[G]$ for all characters $\chi$ of $G$. The proposition now follows.
\end{proof}

Since $E_d$ is a finite index subgroup of $E$, it follows that $( \mathscr{E} / \mathscr{E}_d )_{\chi}$ is finite for every $\chi \neq 1$. Let $\varepsilon_d'$ be a generator for the $\mcal{O}[G]$-ideal $\mathscr{E}_d$. Since $\mathscr{E} = (\varepsilon) \supseteq (\varepsilon_d') = \mathscr{E}_d$, there exists $\varepsilon_d \in \mcal{O}[G]$ such that $\varepsilon \cdot \varepsilon_d = \varepsilon_d'$. So $\mathscr{E}_d = \varepsilon_d \mathscr{E}$, what's more,
\[ [\mathscr{E}: \mathscr{E}_d] = \prod_{\chi \neq 1} \left| \chi(\varepsilon_{d}) \right|_p^{-e(p) f(p)}.\]
Combining \Cref{Udm,Cd}, we immediately obtain the following corollary.

\begin{corollary} \label{bigcor} The $\mcal{O}[G]$-ideal $\mathscr{C}(d)$ is generated by $\displaystyle \frac{d}{\bar{d}} \cdot \omega' \cdot \varkappa_{\bar{d}}$, in fact, we have
\[ \mathscr{C}(d) = \frac{d}{\bar{d}} \cdot \varkappa_{\bar{d}} \cdot \omega'  \cdot \varepsilon^{-1} \cdot \varepsilon_d^{-1} \mathscr{E}_d.\]
\end{corollary}

\subsection{Proof of Theorem 1.1}

Let $\rho$ be as in the statement of the theorem. Note that
\[ |\Syl_p(E_d/C_d)_{\rho}| = |\Syl_p(\Cl_d)_{\rho}| \quad \text{iff} \quad [\mathscr{E}_{d,\rho}:\mathscr{C}_{d,\rho}] = \left| \left( \Cl_d \otimes \mcal{O} \right)_{\rho} \right|.\]
We aim to prove latter equality. Define the linear transformation
\begin{align*}
\phi_{\rho}: \mathscr{E}_{d,\rho}K &\to \mathscr{C}(d)_{\rho}K \\
 x &\mapsto \frac{d}{\bar{d}} \cdot  \varkappa_{\bar{d}} \cdot \omega' \cdot \varepsilon^{-1} \cdot \varepsilon_d^{-1} x.
 \end{align*}
Note that \Cref{bigcor} gives us that $\phi_{\rho}(\mathscr{E}_{d,\rho}) = \mathscr{C}(d)_{\rho}$. Since $C_d \subseteq C(d)$, it follows that
\[ [\mathscr{E}_{d,\rho}: \mathscr{C}_{d,\rho}] \geq (\mathscr{E}_{d,\rho}:\mathscr{C}(d)_{\rho}) = \left| \det \phi_{\rho} \right|_p^{-e(p) f(p)}.\]
Now, we have
\[ \det \phi_{\rho} = \prod \chi \left( \frac{d}{\bar{d}} \cdot \varkappa_{\bar{d}} \cdot \omega' \cdot \varepsilon^{-1} \cdot \varepsilon_d^{-1} \right), \]
where the product is over those $\chi$ such that $\chi(e_{\rho}) \neq 0$. Note that
\[ \chi \left( \frac{d}{\bar{d}} \cdot \varkappa_{\bar{d}} \cdot \omega' \cdot \varepsilon^{-1} \cdot \varepsilon_d^{-1} \right) = \left[ \frac{d}{\bar{d}} \cdot \prod_{\ell \mid \bar{d}} \big( \ell - \chi(\sigma_{\ell}e_{\ell}) \big) \right] \cdot \left[ L_p'(0,\chi) \chi(\varepsilon)^{-1} \right] \cdot \left[ \chi(\varepsilon_{d})^{-1} \right]. \]
Let $\mathscr{O}_d^{\times} = \big( \mfrak{o}/d \big)^{\times} \otimes \mcal{O}$. Extending scalars in \Cref{rays}, we get that
\[ \mathscr{O}_d^{\times} \simeq \prod_{\ell \mid d} \mathscr{O}_{\ell}^{\times}, \quad \text{where} \quad \mathscr{O}_{\ell}^{\times} \simeq \begin{cases} \bigslant{\mcal{O}[G]}{\big( \ell-\sigma_{\ell} \cdot e_{\ell} \big)} & \ell \neq p \\
\bigslant{\mcal{O}[G]}{\big( p^e-p^{e-1} \sigma_p \cdot e_p \big)} & \ell =p.
\end{cases}\]
It follows that $[\mathscr{O}_{d,\chi}^{\times}: 1] = \left| (d/\bar{d}) \prod_{\ell \mid \bar{d}} \big( \ell - \chi(\sigma_{\ell} \cdot e_{\ell}) \big) \right|_p^{-e(p) f(p)}$, and from \Cref{padicsinnottcor}, we have $[\mathscr{E}_{\chi}:\mathscr{C}_{\chi}] = \left| L_p'(0,\chi) \chi(\varepsilon)^{-1} \right|_p^{-e(p) f(p)}$. Hence
\begin{equation} \label{rhoparts} [\mathscr{E}_{d,\rho}: \mathscr{C}_{d,\rho}] \geq (\mathscr{E}_{d,\rho}:\mathscr{C}(d)_{\rho}) =  \frac{ [\mathscr{O}_{d,\rho}^{\times}:1] \cdot [\mathscr{E}_{\rho}: \mathscr{C}_{\rho}]}{ [\mathscr{E}_{\rho}:\mathscr{E}_{d,\rho}]} \end{equation}
Now, recall the ``unscrewing'' of $\Cl(\mfrak{a}) =I(\mfrak{a})/P_{\mfrak{a}}$ where $\mfrak{a}$ is an ideal of $\mfrak{o}$:

\begin{multicols}{2}
\begin{tikzpicture}[scale=.9]
\draw (0,0) node (Cl) {$\Cl$};
\draw (0,-1.5) node (1) {$1$};
\draw (2,0) node (Id) {$I(\mfrak{a})$};
\draw (2,-1.5) node (Pd) {$P(\mfrak{a})$};
\draw (2,-3) node (Pdd) {$P_{\mfrak{a}}$};
\draw (4,-1.5) node (kdE) {$k^{\times}(\mfrak{a}) E$};
\draw (4,-3) node (kddE) {$k_{\mfrak{a}}^{\times} E$};
\draw (4,-4.5) node (kdd) {$k_{\mfrak{a}}^{\times}$};
\draw (6,-3) node (E) {$E$};
\draw (6,-4.5) node (Ed) {$E_{\mfrak{a}}$};

\draw [<-] (Cl) -- (Id);
\draw [->] (Pd) -- (1);
\draw [->] (Pd) -- (kdE);
\draw [->] (Pdd) -- (kddE);
\draw [->] (kddE) -- (E);
\draw [->] (kdd) -- (Ed);

\draw (Cl) -- (1);
\draw (Id) -- (Pd) -- (Pdd);
\draw (kdE) -- (kddE) -- (kdd);
\draw (E) -- (Ed);
\end{tikzpicture}

\begin{itemize}[leftmargin=*]
\item $I(\mfrak{a})$: fractional ideals co-prime to $d$
\item $P(\mfrak{a})$: principal fractional ideals co-prime to $p$
\item $k^{\times}(\mfrak{a})$: elements of $k^{\times}$ co-prime to $\mfrak{a}$
\item $k_{\mfrak{a}}^{\times}$:  elements of $k^{\times}$ congruent to $1$ modulo $\mfrak{a}$.
\item $P_{\mfrak{a}}$:  principal fractional ideals generated by elements of $k_a^{\times}$
\end{itemize}

\end{multicols}

Since \Cref{rhoparts} holds for all $\rho \neq 1$ with $[\mathscr{E}_d:\mathscr{C}_d] = \left| \Cl_d \otimes \mcal{O} \right|$ and $k^{\times}(d) E/k_d^{\times} \simeq \left( \mfrak{o}/d \right)^{\times}$, taking the product over all such $\rho$ we get
\begin{align*}
|\Cl_d \otimes \mcal{O}| = [\mathscr{E}_d:\mathscr{C}_d] &\geq_p \frac{ [(1-e_1)\mathscr{O}_{d}^{\times}:1] \cdot [\mathscr{E}: \mathscr{C}]}{ [\mathscr{E}:\mathscr{E}_{d}]} \\
&= \frac{ [(1-e_1)\mathscr{O}_{d}^{\times}:1] \cdot | \Cl \otimes \mcal{O}|}{ [\mathscr{E}:\mathscr{E}_{d}]} \\
&= | \Cl_d \otimes \mcal{O}|.
\end{align*}
In light of \Cref{rhoparts}, it follows that
\[ [\mathscr{E}_{d,\rho}: \mathscr{C}_{d,\rho}] = \frac{ [\mathscr{O}_{d,\rho}^{\times}:1] \cdot [\mathscr{E}_{\rho}: \mathscr{C}_{\rho}]}{ [\mathscr{E}_{\rho}:\mathscr{E}_{d,\rho}]}.\]
But $[\mathscr{E}_{\rho}: \mathscr{C}_{\rho}] = | \left( \Cl \otimes \mcal{O} \right)_{\rho} |$, so it follows that
\[ [\mathscr{E}_{d,\rho}:\mathscr{C}_{d,\rho}] = \frac{ [\mathscr{O}_{d,\rho}^{\times}:1] \cdot \left| \left( \Cl \otimes \mcal{O} \right)_{\rho} \right|}{ [\mathscr{E}_{\rho}:\mathscr{E}_{d,\rho}]} = \left| \left( \Cl_d \otimes \mcal{O} \right)_{\rho} \right|.\]
This proves \Cref{main}.

\section{Gauss Sums and Stickelberger's Theorem for Ray Class Groups}

We now relax the condition on $p$ so that the results in this section are applicable to situations when $p \mid [k:\mbb{Q}]$. The main goal is to prove \Cref{rayrubin}, a ray class version of a theorem of Rubin \cite[Theorem 1.3]{Rubin} (specialized to the case when the base field is $\mbb{Q}$) which itself generalized a theorem of Thaine \cite[Theorem 6]{Thaine}.

The following lemma will act as an explicit version of Hilbert's Theorem 90 for our purposes.
\begin{lemma} \label{explicith90}
Let $\ell$ be a rational prime completely split in $k$. For any $\epsilon \in \mfrak{o}_{k(\zeta_{\ell})}^{\times}$ such that $N^{k(\zeta_{\ell})}_k(\epsilon) = 1$, we have that the element
\[ \alpha(\ell,\epsilon)=\alpha := - \sum_{a=1}^{\ell-1} \zeta_{\ell}^{\tau^a} \epsilon^{1+\tau+\cdots + \tau^{a-1}}= - \zeta_{\ell}^{\tau}\epsilon - \zeta_{\ell}^{\tau^2} \epsilon^{1+\tau} - \cdots - \zeta_{\ell}^{\tau^{\ell-1}} \epsilon^{1 + \tau + \cdots + \tau^{\ell-2}} \]
is non-zero for some choice of $\zeta_{\ell}$, moreover, $\alpha^{1-\tau} = \epsilon$ where $\langle \tau \rangle = \Gal(k(\zeta_{\ell})/k)$.
\end{lemma}
\begin{proof}
Let $\alpha(x) \in \mbb{C}(x)$ be the rational function defined by
\[ x \mapsto -\sum_{a=1}^{\ell-1} \frac{\zeta_{\ell}^{\tau^a}}{1-x\zeta_{\ell}^{\tau^a}} \cdot \epsilon^{1+\tau+\cdots +\tau^{a-1}}.\]
Since $\alpha(x)$ has distinct poles, it follows that $\alpha(x)$ is not identically zero. On the other hand, we may view $\alpha(x) \in \mbb{C}\llbracket x \rrbracket$ and write
\[ \alpha(x) = \sum_{n=0}^{\infty} \big(-\zeta_{\ell}^{(n+1)\tau} \epsilon - \zeta_{\ell}^{(n+1)\tau^2} \epsilon^{1+\tau} - \cdots - \zeta_{\ell}^{(n+1)\tau^{\ell - 1}} \epsilon^{1 + \tau + \cdots + \tau^{\ell-2}} \big) x^n.\]
Note that the power series form of $\alpha(x)$ has periodic coefficients of the form of the claim. Since $\alpha(x)$ is not identically zero, we get that $\alpha \neq 0$ for some choice of $\zeta_{\ell}$. In fact, $\alpha \neq 0$ for at least two choices of $\zeta_{\ell}$ for if otherwise, then $\alpha(x)$ has a pole at $x=1$, a contradiction. This proves the first claim.

Now, notice that
\begin{align*}
\epsilon \alpha^{\tau} &= -\zeta_{\ell}^{\tau^2} \epsilon^{1+\tau} - \zeta_{\ell}^{\tau^3} \epsilon^{1+\tau+\tau^2} - \cdots - \zeta_{\ell}^{\tau^{\ell}} \epsilon^{1+\tau + \cdots + \tau^{\ell-1}} \\
& = - \zeta_{\ell}^{\tau^2} \epsilon^{1+\tau} - \zeta_{\ell}^{\tau^3} \epsilon^{1+\tau+\tau^2} - \cdots - \zeta_{\ell}^{\tau}\epsilon \\
&= \alpha,
\end{align*}
since $\tau^{\ell} = \tau$ and $1 + \tau + \cdots + \tau^{\ell-1} = 1$. This proves the lemma (alternatively see \cite{All3}).
\end{proof}

Fix an ideal $\mfrak{a} \subseteq \mfrak{o}$. For odd primes $\ell$ that are completely split in $k$, let $k(\zeta_{\ell})^{\times}(\mfrak{a})$ denote the set of all elements of $k(\zeta_{\ell})^{\times}$ that are co-prime to $\mfrak{a}$ and define
\begin{align*}
N_{\ell,\mfrak{a}} : k(\zeta_{\ell})^{\times}(\mfrak{a}) &\to \big( \bigslant{\mfrak{o}}{\mfrak{a}} \big)^{\times} \\
x & \mapsto N^{k(\zeta_{\ell})}_k(x) \bmod{\mfrak{a}}.
\end{align*}
Now set
\[ E(\ell,\mfrak{a}) := \left\{ \epsilon \in E_{k(\zeta_{\ell})} : N^{k(\zeta_{\ell})}_k(\epsilon) = 1,\   N_{\ell,\mfrak{a}} \big( \alpha(\ell,\epsilon) \big) \in \im\left( \bigslant{E}{E_{\mfrak{a}}} \to \left( \bigslant{\mfrak{o}}{\mfrak{a}} \right)^{\times} \right) \right\}. \]
The following theorem is a generalization of \cite[Theorem 5.1]{Rubin}.

\begin{thm}\label{thaine}
Let $n \in \mathbb{N}$ and let $\ell$ be an odd prime split completely in $k$ such that $\ell \equiv 1 \mod{n}$ and $(\ell)$ is co-prime to $\mfrak{a}$. Fix a prime $\lambda$ of $k$ above $\ell$, and let $\mcal{A} \subseteq (\mbb{Z}/n\mbb{Z})[G]$ be the annihilator of the cokernel of the natural map
\[ \phi: E(\ell, \mfrak{a}) \to \left( \bigslant{\mfrak{o}_{k(\zeta_{\ell})}}{L} \right)^{\times} \otimes \bigslant{\mbb{Z}}{n\mbb{Z}}, \]
where $L$ is the product of all primes of $\mfrak{o}_{k(\zeta_{\ell})}$ above $\ell$. Then $\mcal{A}$ annihilates the class of $\lambda$ in $\Cl(\mfrak{a})/n\Cl(\mfrak{a})$.
\end{thm}
\begin{proof}
Let $\theta \in \mcal{A}$, and let $u \in \mfrak{o}_{k(\zeta_{\ell})}$ such that
\[ u \equiv s^{-1} \bmod{\mcal{L}} \qquad \text{and} \qquad u \equiv 1 \bmod{\mcal{L}^{\sigma}} \quad \text{ for all } \sigma \neq \id,\]
where $\mcal{L}$ is the prime of $\mfrak{o}_{k(\zeta_{\ell})}$ above $\lambda$ and $\langle s \rangle = \mbb{Z}/\ell\mbb{Z}^{\times}$. The element $u$ has been chosen so that
\[ \left( \mfrak{o}_{k(\zeta_{\ell})} / L \right)^{\times} = \langle u \bmod{L} \rangle_{\mbb{Z}/(\ell-1)\mbb{Z} [G]}.\]
Now, $u^{\theta} \equiv \eta^n \epsilon \bmod{L}$ for some $\eta \in k(\zeta_{\ell})^{\times}$ coprime to $\ell$ and $\epsilon \in E(\ell, \mfrak{a})$. Let $\tau$ be a generator for $\Gal(k(\zeta_{\ell})/k)$, and $\alpha = \alpha(\ell,\epsilon)$ be as in \Cref{explicith90}.
Now, $(\alpha)$ is a non-zero ideal inert under the action imposed by $\Gal(k(\zeta_{\ell})/k)$. It follows that there exists an ideal $\mfrak{b}\subseteq \mfrak{o}_k$ satisfying
\[ (\alpha) = \mfrak{b} \cdot \prod_{\sigma \in G} \mcal{L}^{a_{\sigma} \sigma^{-1}}, \]
where no conjugate of $\mcal{L}$ is supported by $\mfrak{b}$. Taking norms of both sides of the above we get
\[ \left( N^{k(\zeta_{\ell})}_k (\alpha) \right) = \mfrak{b}^{\ell -1} \cdot \lambda^{\sum a_{\sigma} \sigma^{-1}}. \]
Since $N_{\ell,\mfrak{a}} \big( \alpha \big) \in \im\left( \bigslant{E}{E_{\mfrak{a}}} \to \left( \bigslant{\mfrak{o}}{\mfrak{a}} \right)^{\times} \right)$, we have that $\big( N^{k(\zeta_{\ell})}_k (\alpha) \big) \in P_{\mfrak{a}}$. By assumption we have $n \mid (\ell -1)$, so $\sum a_{\sigma} \sigma^{-1} \bmod{n \mathbb{Z}[G]}$ annihilates the class of $\lambda$ in $\Cl(\mfrak{a})/n \Cl(\mfrak{a})$.

It remains to relate the coefficients $a_{\sigma}$ to $\theta$. To that end, note that
\[ a_{\sigma} = \ord_{\mcal{L}^{\sigma^{-1}}} (\alpha) = \ord_{\mcal{L}^{\sigma^{-1}}} (1-\zeta_{\ell})^{a_{\sigma}}. \]
Write $\alpha = \beta(1-\zeta_{\ell})^{a_{\sigma}}$ where $\beta$ is a $\mcal{L}^{\sigma^{-1}}$-unit. Without loss of generality, let's suppose $\tau: \zeta_{\ell} \to \zeta_{\ell}^s$. The primes above $\ell$ are totally ramified in $k(\zeta_{\ell})/k$. So $\tau$ acts trivially on $\mcal{L}^{\sigma^{-1}}$-units modulo $\mcal{L}^{\sigma^{-1}}$. Hence
\begin{align*}
\epsilon = \frac{\alpha}{\alpha^{\tau}} &= \frac{\beta(1-\zeta_{\ell})^{a_{\sigma}}}{\beta^{\tau}(1-\zeta_{\ell}^{\tau})^{a_{\sigma}}} \\
&\equiv \left( \frac{1-\zeta_{\ell}}{1-\zeta_{\ell}^{\tau}} \right)^{a_{\sigma}} \bmod{\mcal{L}^{\sigma^{-1}}} \\
&\equiv (s^{-1})^{a_{\sigma}} \bmod{\mcal{L}^{\sigma^{-1}}}.
\end{align*}
This gives us that $\epsilon \equiv u^{a_{\sigma} \sigma^{-1}} \bmod{\mcal{L}^{\sigma^{-1}}}$, so
\[ \epsilon \equiv u^{\sum a_{\sigma} \sigma^{-1}} \equiv \eta^{-n} u^{\theta} \mod{L}. \]
Hence $\sum a_{\sigma} \sigma^{-1} \equiv \theta \bmod{n \mbb{Z}[G]}$.
\end{proof}

\begin{remark} If in \Cref{explicith90} we take $k = \mbb{Q}(\zeta_m)$ and $\ell \equiv 1 \bmod{m}$ with $\epsilon = \zeta_m$, then $\alpha$ is the classical Gauss sum. In this case, we have that
\[ \big( \alpha^{\ell-1} \big) = \lambda^{\sum a_{\sigma} \sigma^{-1}} \]
where, similar to \Cref{thaine}, we have $\zeta_m \equiv u^{a_{\sigma} \sigma^{-1}} \bmod \lambda^{\sigma^{-1}}$. The dependence of the coefficients $a_{\sigma}$ on $\ell$ is easy to tease out of this congruence, and we're a hop, skip, and a jump away from the classical Stickelberger theorem. For the more general types of elements $\alpha$ in \Cref{explicith90} and \Cref{thaine}, the dependence of the $a_{\sigma}$ on $\ell$ is more difficult to separate. Instead of reckoning with this obstacle, we step around it and show that any $G$-module map from $E/E^{p^n}$ to $\mbb{Z}/p^n\mbb{Z}[G]$ can be effectively filtered through $\big( \mfrak{o}_{k(\zeta_{\ell})}/L \big)^{\times} \otimes \mbb{Z}/p^n\mbb{Z}$ for certain well chosen primes $\ell$. This idea was first employed by Rubin in \cite{Rubin}.
\end{remark}

\Cref{thaine} inspires us to make the following definition.
\begin{defin} \label{aspecial}
For an ideal $\mfrak{a} \subseteq \mfrak{o}$, let $\mcal{T}(\mfrak{a})$ denote the set of numbers $\varsigma \in k^{\times}$ such that for all but finitely many primes $\ell$ split completely in $k$, we have that there is an $\epsilon \in E(\ell,\mfrak{a})$ such that for all $\sigma \in G$,
\[ \epsilon \equiv \varsigma \mod \mcal{L}^{\sigma} \]
where $\mcal{L} \subset \mfrak{o}_{k(\zeta_{\ell})}$ is a prime ideal such that $\mcal{L} \mid \ell$. We call $\mcal{T}(\mfrak{a})$ the $\mfrak{a}$-special numbers of $k$. Let
\[ \mcal{S}(\mfrak{a}) := \mcal{T}(\mfrak{a}) \cap E.\]
We call $\mcal{S}(\mfrak{a})$ the $\mfrak{a}$-special units of $k$.
\end{defin}

The $1$-special numbers are, in fact, Rubin's special numbers from \cite{Rubin}. It's fair to ask if $\mfrak{a}$-special numbers even exist. For an appropriate choice of $d$, the following theorem will show that Schmidt's $d$-cyclotomic units are contained in the $\mfrak{a}$-special units. So $\mcal{S}(\mfrak{a})$ is a subgroup of finite index of $E$.

\begin{thm} \label{explicitspecial}
If $\delta \in D\big(d\big)$, then $\pm \delta \in \mcal{T}(d)$, i.e., $\pm D\big(d\big) \subseteq \mcal{T}(d)$.
\end{thm}
\begin{proof}
It suffices to show that $\pm \delta_{n,d} \in \mcal{T}(d)$ for all $n>1$ and $n \nmid d$ since these numbers generate $D(d)$. Let $\ell$ be a rational prime split completely in $k$ such that $(\ell,nd) =1$. Define
\[ \pm \epsilon_{n,d} = N^{\mbb{Q}(\zeta_{n\ell})}_{k^n(\zeta_{\ell})} \prod_{t \mid \overline{d}} (\zeta_{\ell}^t - \zeta_n^t)^{\mu(t) d/t} \in k(\zeta_{\ell}). \]
Let $\lambda$ be a prime of $k$ above $\ell$ and $\mcal{L}$ the prime of $k(\zeta_{\ell})$ above $\lambda$. Since
\[ (1-\zeta_{\ell}) \mfrak{o}_{k(\zeta_{\ell})} = \prod_{\sigma \in G} \mcal{L}^{\sigma}, \]
it follows that $\zeta_{\ell} \equiv 1 \bmod{\mcal{L}^{\sigma}}$ for all $\sigma \in G$, hence $\pm \epsilon_{n,d} \equiv \pm \delta_{n,d} \bmod{\mcal{L}^{\sigma}}$ for all $\sigma \in G$. Now, we note
\begin{align*}
N^{k^n(\zeta_{\ell})}_{k^n} (\pm \epsilon_{n,d}) &= N^{\mbb{Q}(\zeta_n)}_{k^n} N^{\mbb{Q}(\zeta_{n\ell})}_{\mbb{Q}(\zeta_n)} \prod_{t \mid \overline{d}} (\zeta_{\ell}^t - \zeta_n^t)^{\mu(t) d/t} \\
&= N^{\mbb{Q}(\zeta_n)}_{k^n} \prod_{t \mid \overline{d}} \left( \frac{ \zeta_n^{t\ell} - 1}{ \zeta_n^t - 1} \right)^{\mu(t) d/t} \\
&= \delta_{n,d}^{\sigma_{\ell} - 1},
\end{align*}
where $\sigma_{\ell}$ is the Frobenius automorphism for $\ell$ in $k$. Since $\ell$ splits completely in $k$, it follows that $\sigma_{\ell} =1$ hence $N^{k_n(\zeta_{\ell})}_{k^n} (\epsilon_{n,d}) = N^{k(\zeta_{\ell})}_k (\epsilon_{n,d}) = 1$.

Now, let $q \mid d$ be a prime, let $q^j$ be the $q$-primary part of $d$, and let $d_q = d/q^j$. Then
\[ \epsilon_{n,d} = N^{\mathbb{Q}(\zeta_{n\ell})}_{k^n(\zeta_{\ell})} \prod_{t \mid \frac{\overline{d}}{q}} \left[ \frac{(\zeta_{\ell}^t - \zeta_n^t)^q}{(\zeta_{\ell}^{tq} - \zeta_n^{tq})} \right]^{q^{j-1} \mu(t) d_q/t}.\]
For all $t \mid \overline{d}/q$ we have that $\zeta_{\ell}^t$ and $\zeta_{\ell}^{tq}$ are primitive $\ell$-th roots of unity since $(\ell,nd)=1$, moreover, $\zeta_n^t$ and $\zeta_n^{tq}$ are not equal to $1$ since $n \nmid \overline{d}$. It follows that  $(\zeta_{\ell}^t - \zeta_n^t)$ and $(\zeta_{\ell}^{tq} - \zeta_n^{tq})$ are units in $\mbb{Z}[\zeta_{n\ell}]$, hence $\epsilon_{n,d} \in \mfrak{o}_{k(\zeta_{\ell})}^{\times}$. We also have
\[ \frac{(\zeta_{\ell}^t - \zeta_n^t)^q}{\zeta_{\ell}^{tq} - \zeta_n^{tq}} \equiv 1 \bmod{q}\]
from which it follows that
\[ \left[ \frac{(\zeta_{\ell}^t - \zeta_n^t)^q}{\zeta_{\ell}^{tq} - \zeta_n^{tq}} \right]^{q^{j-1}} \equiv 1 \bmod{q^j}, \]
whence $\epsilon_{n,d} \equiv 1 \bmod{q^j}$. Since $q$ was an arbitrary divisor of $d$, it follows that $\epsilon_{n,d} \equiv 1 \bmod{d}$, hence $\pm \epsilon_{n,d} \equiv \pm 1 \bmod{d}$.

Now, let $N_b = 1 + \tau + \cdots + \tau^{b-1}$. Note that
\begin{align*}
\alpha(\ell,\pm \epsilon_{n,d}) &= \mp \sum_{b=1}^{\ell -1} \zeta_{\ell}^{\tau^b} \epsilon_{n,d}^{N_b} \\
&\equiv \mp \sum_{b=1}^{\ell-1} \zeta_{\ell}^{\tau^b} \bmod{d} \\
&\equiv \mp 1 \bmod{d}.
\end{align*}
So $N^{k(\zeta_{\ell})}_k \big( \alpha(\ell_,\pm \epsilon_{n,d}) \big) \equiv 1 \bmod{d}$. This proves that $\pm \epsilon_{n,d} \in E(\ell,d)$ as claimed.
\end{proof}

Now, let
\begin{align*}
A_{n}(\mfrak{a}) &= \Cl(\mfrak{a})/p^n \Cl(\mfrak{a}) \\
A(\mfrak{a}) &= \Syl_p(\Cl(\mfrak{a})) \\
F_{n}(\mfrak{a}) &= \text{the ray class field over $k$ associate to $A_{n}(\mfrak{a})$. } \\
F(\mfrak{a}) &= \text{the ray class field over $k$ associate to $A(\mfrak{a})$.}
\end{align*}
Note that
\[ \Gal(F_n(\mfrak{a})/k) \simeq A_n(\mfrak{a}) \quad \text{and} \quad \Gal(F(\mfrak{a})/k) \simeq A(\mfrak{a}) \]
via the Artin map. Set $A_n'(\mfrak{a}) \leq A_n(\mfrak{a})$ such that $A_n'(\mfrak{a}) \simeq \Gal(F_n(\mfrak{a})/(F_n(\mfrak{a}) \cap k(\zeta_{p^n})))$.

\begin{thm} \label{mainthm}
Let $\alpha: E/E^{p^n} \to \mbb{Z}/p^n\mbb{Z}[G]$ be a $G$-module map. Then
\[ \alpha \left( \mcal{S}(\mfrak{a}) E^{p^n}/E^{p^n} \right) \quad \text{annihilates} \quad A_n'(\mfrak{a}). \]
\end{thm}
\begin{proof}
The argument here is essentially the same as in \cite{Rubin} (with base field $\mathbb{Q}$) but with some natural adjustments, so we give a somewhat abbreviated version of the proof. Let $\mcal{G} = \Gal \big( k(\zeta_{p^n})/\mbb{Q}\big)$ and let
\[ \Gamma = \Gal \left( \bigslant{ k(\zeta_{p^n}, E^{1/p^n})}{ k( \zeta_{p^n}, (\ker \alpha)^{1/p^n})} \right).\]
Let $\mcal{G}$ act on $\Gamma$ by conjugation, and let $\gamma_1,\ldots, \gamma_j$ be a complete system of unique representatives of $\Gamma/\mcal{G}$. Note that $F_n(\mfrak{a})$ and $k(\zeta_{p^n},E^{1/p^n})$ are linearly disjoint over $F_n(\mfrak{a}) \cap k(\zeta_{p^n})$. This is because Kummer theory gives us an isomorphism of $\Gal(k(\zeta_{p^n})/k)$-modules
\[ \Gal \left( \bigslant{k(\zeta_{p^n}) (F_n(\mfrak{a}) \cap k(\zeta_{p^n}, E^{1/p^n}))}{ k(\zeta_{p^n}) } \right) \simeq \Hom_{\mathbb{Z}}(B,\mu_{p^n}), \]
where $B \leq E/E^{p^n}$ such that $k(\zeta_{p^n},E^{1/p^n}) = k(B^{1/p^n})$ and $\mu_{p^n}$ is the group of $p^n$-th roots of unity. The Galois group on the left is abelian, so it follows that $\Gal(k(\zeta_{p^n})/k)$ acts trivially on $\Hom_{\mathbb{Z}}(B,\mu_{p^n})$. Since $k$ is real and $p$ is an odd prime, it must be that $\Hom_{\mathbb{Z}}(B,\mu_{p^n})$ is trivial whence the claim that $F_n(\mfrak{a})$ and $k(\zeta_{p^n},E^{1/p^n})$ are linearly disjoint over $F_n(\mfrak{a}) \cap k(\zeta_{p^n})$.

Now, let $\mfrak{c} \in A_n'(\mfrak{a})$. We may choose $\beta_i \in \Gal\big( F_n(\mfrak{a}) k(\zeta_{p^n}, E^{1/p^n})/k \big)$ such that
\[ \beta_i|_{F_n(\mfrak{a})} = \mfrak{c} \qquad \text{and} \qquad \beta_i|_{k(\zeta_{p^n}, E^{1/p^n})} = \gamma_i.\]
By the Chebotarev Density Theorem, there exists infinitely many degree $1$ non-conjugate $j$-tuples of primes $\lambda_1,\ldots, \lambda_j \subseteq \mfrak{o}_k$ such that $(\lambda_i,\mfrak{a}) =1$ and $\beta_i$ is in the conjugacy class of Frobenius automorphisms for $\lambda_i$ in $\Gal\big( F_n(\mfrak{a})k(\zeta_{p^n},E^{1/p^n})/k \big)$. It follows that $\lambda_i \in \mfrak{c}$. We let $\ell_i$ be the rational prime below $\lambda_i$. Since $\beta_i|_{k(\zeta_{p^n})} = \id$, it follows that $\ell_i \equiv 1 \bmod{p^n}$.

Now, let $\phi$ be the natural map from $E/E^{p^n} \to (\mfrak{o/L})^{\times} \otimes \mbb{Z}/p^n\mbb{Z}$ where $\mfrak{L} = \prod_{i=1}^j \ell_i$. From the exact sequence of $\Gal(k(\zeta_{p^n})/k)$-modules
\[ 1 \to \mu_{p^n} \to k(\zeta_{p^n})^{\times} \xrightarrow{p^n} k(\zeta_{p^n})^{\times p^n} \to 1 \]
we obtain the exact sequence of $\Gal(k(\zeta_{p^n})/k)$-invariants
\[ 1 \to k^{\times} \xrightarrow{p^n} k(\zeta_{p^n})^{\times p^n} \cap k \to 1.\]
So $[k(\zeta_{p^n})^{\times p^n} \cap k : k^{\times p^n}] =1$. So we have
\begin{align*}
\epsilon \in \ker \alpha \quad &\text{iff $\epsilon^{1/p^n} \in k(\zeta_{p^n}, (\ker \alpha)^{1/p^n})$}, \\
&\text{iff $g \gamma_i$ fixes $k(\zeta_{p^n},\epsilon^{1/p^n})$ for all $g \in \mcal{G}$, $\gamma_i \in \Gamma/\mathcal{G}$}, \\
&\text{iff $\lambda_i^{\sigma}$ splits completely in $k(\epsilon^{1/p^n})$ for all $\sigma \in G$, $i=1,\ldots,j$},\\
&\text{iff $\epsilon \in \ker \phi$}.
\end{align*}
In short, since $\Gamma$ is saturated with Frobenius automorphisms for the $\lambda_i$, it follows that $\epsilon \in \ker \alpha$ if and only if $\epsilon \in \ker \phi$. This allows us to consider the well-defined map
\[ \alpha \circ \phi^{-1}: \im(\phi) \to \mbb{Z}/p^n\mbb{Z}[G]\]
which we may lift to a map $f: (\mfrak{o}/\mfrak{L})^{\times} \otimes \mbb{Z}/p^n\mbb{Z} \to (\mbb{Z}/p^n\mbb{Z})[G]$ obtaining the following commutative diagram:
\begin{center}
\begin{tikzpicture}[scale=.75]
\node (E) at (0,2) {$E/E^{p^n}$};
\node (O) at (3,2) {$(\mbb{Z}/p^n\mbb{Z})[G]$};
\node (T) at (0,0) {$ (\mfrak{o}/\mfrak{L} )^{\times} \otimes \mbb{Z}/p^n \mbb{Z}$};

\draw [->] (E) to node [left] {$\phi$} (T);
\draw [->] (T) to node [right,below] {$f$} (O);
\draw [->] (E) to node [above] {$\alpha$} (O);
\end{tikzpicture}
\end{center}
Now, let $\varsigma_{\mfrak{a}} \in \mcal{S}(\mfrak{a})$. Without loss of generality, we may assume that for each $i$, there exists $\epsilon_i \in E(\ell_i, \mfrak{a})$ such that
\[ \epsilon_i \equiv \varsigma_{\mfrak{a}} \bmod{\mcal{L}_i^{\sigma}} \quad \text{for all } \sigma \in G,\]
where $\mcal{L}_i \subset \mfrak{o}_{k(\zeta_{\ell_i})}$ is the prime above $\lambda_i$. Set
\[ L_i := \prod_{\sigma \in G} \mcal{L}_i^{\sigma}. \]
Since the primes of $\mfrak{o}$ above $\ell_i$ are totally ramified in $k(\zeta_{\ell_i})$, we have that
\[ \left( \mfrak{o}/\mfrak{L} \right)^{\times} \simeq  \prod_{i=1}^j (\mfrak{o}/\ell_i)^{\times} \simeq \prod_{i=1}^j \left( \mfrak{o}_{k(\zeta_{\ell_i})}/L_i \right)^{\times}. \]
This association allows us to consider $\phi$ and $f$ as functions defined into and on
\[ \prod \left( \mfrak{o}_{k(\zeta_{\ell_i})}/L_i \right)^{\times} \otimes \mbb{Z}/p^n\mbb{Z},\] respectively. Let $u_i \in \mfrak{o}_{k(\zeta_{\ell_i})}$ such that
\[ u_i \equiv s_i^{-1} \bmod{\mcal{L}_i} \qquad \text{and} \qquad u_i \equiv 1 \bmod{\mcal{L}_i^{\sigma}} \quad \text{ for all } \sigma \in G, \sigma \neq \id, \]
where $\langle s_i \rangle = \bigslant{\mbb{Z}}{\ell_i\mbb{Z}}^{\times}$, as in \Cref{thaine}. Note that
\[ \langle u_i \bmod{L_i} \rangle_{(\mbb{Z}/(\ell_i-1)\mbb{Z})[G]} \simeq \left( \mfrak{o}_{k(\zeta_{\ell_i})}/L_i \right)^{\times}, \]
Let $\theta_i \in (\mbb{Z}/p^n\mbb{Z})[G]$ such that
\[ \mfrak{s}_{\mfrak{a}} = \mfrak{u}_i^{\theta_i}, \quad \text{where} \quad \begin{cases} \mfrak{s}_{\mfrak{a}} = \big( \varsigma_{\mfrak{a}} \bmod{L_i} \big) \otimes 1 \\
\mfrak{u}_i = \big( u_i \bmod{L}_i \big) \otimes 1 \end{cases} \in \left( \mfrak{o}_{k(\zeta_{\ell_i})}/L_i \right)^{\times} \otimes \bigslant{\mbb{Z}}{p^n \mbb{Z}}\]
Since $\varsigma_{\mfrak{a}} \bmod{L_i} \equiv \epsilon_i \bmod{L}_i$, it follows that $\theta_i$ is an annihilator of the cokernel of the map
\[ E(\ell_i, \mfrak{a}) \to \left( \mfrak{o}_{k(\zeta_{\ell_i})}/L_i \right)^{\times} \otimes \bigslant{\mbb{Z}}{p^n \mbb{Z}}. \]
So $\theta_i$ annihilates the class of $[\lambda_i]=\mfrak{c}$ in $\Cl(\mfrak{a})/p^n \Cl(\mfrak{a})$ by \Cref{thaine}. Let $\overline{\varsigma_{\mfrak{a}}} = \varsigma_{\mfrak{a}} \bmod{E^{p^n}}$. Then
\[ \alpha(\overline{\varsigma_{\mfrak{a}}}) = f(\mfrak{s}_{\mfrak{a}})= \sum_{i=1}^j \theta_i \cdot f(\mfrak{u}_i),\]
and we get
\[ \mfrak{c}^{\alpha(\overline{\varsigma_{\mfrak{a}}})} = \prod_{i=1}^j ( \mfrak{c}^{\theta_i} )^{f (\mfrak{u}_i) } = \prod_{i=1}^j 1^{f(\mfrak{u}_i)} = 1.\]
This completes the proof of the theorem.
\end{proof}
We may now give a proof of \Cref{rayrubin}
\begin{proof}[Proof of \Cref{rayrubin}]
Let $\omega_1,\omega_2,\ldots, \omega_t$ be a $\mbb{Z}_p$-basis for $\mcal{O}$. For $i=1,\ldots,t$ define $\alpha_i: E \to \mbb{Z}_p[G]$ by
\[ \alpha(\epsilon) = \sum_{i=1}^t \alpha_i(\epsilon) \omega_i.\]
Each $\alpha_i$ is a $G$-module map, and since $\Syl_p(\mfrak{C}_{\mfrak{a}}) \leq A_n'(\mfrak{a})$, the corollary now follows from \Cref{mainthm}.
\end{proof}
We now prove \Cref{raysolomon}.
\begin{proof}[Proof of \Cref{raysolomon}] Let $\vartheta$ be the $G$-module map defined in the previous section. For every $\alpha \in \Hom_G(E,\mcal{O}[G])$, there exists $\beta \in K[G]$ such that $\alpha = \beta \vartheta$ by \cite[Theorem 5.1]{All1}. So for every $\beta \in K[G]$ such that $\beta \vartheta(E) \subseteq \mcal{O}[G]$, $\beta \vartheta(\mcal{S}(\mfrak{a}))$ annihilates $\mfrak{C}_{\mfrak{a}} \otimes \mcal{O}$. If the ramification index of $p$ in $k$ is $e(p) = p^j b$ where $(p,b)=1$, then we could take
\[ \beta =\varpi^{j e(p) - p^j} = \frac{|e(p)|_p^{-1}}{\varpi^{|e(p)|_p^{-1}}},\]
(see \cite[Lemma 2.3]{All1}). In particular, we get the explicit annihilation result: for every $\varsigma_{\mfrak{a}} \in \mcal{S}(\mfrak{a})$ (perhaps a $d$-cyclotomic unit for an appropriate $d$), we have
\[ \frac{|e(p)|_p^{-1}}{\varpi^{|e(p)|_p^{-1}}} \sum_{\sigma \in G} \log_p\big( \varsigma_{\mfrak{a}}^{\sigma} \big) \sigma^{-1} \in \mcal{O}[G] \]
annihilates $\mfrak{C}_{\mfrak{a}} \otimes \mcal{O}$.
\end{proof}

\begin{remark} The above theorem confirms a suspicion of D. Solomon regarding a stronger annihilation result lying beyond his \cite[Conjecture 4.1]{Solomon} (see \cite[Remark 4.1]{Solomon}). In particular, let $U$ denote the maximal abelian pro-p-extension of $k$. The elements of \Cref{raysolomon} are explicit annihilators of explicit quotients of $\Gal(U/k) \otimes \mcal{O}$.
\end{remark}

\begin{remark} The group $\Syl_p\big( E/\mcal{S}(\mfrak{a}) \big)$ is mysterious. In the case when $p \nmid |G|$, the order of $\Syl_p \big( E/\mcal{S}(d) \big)_{\rho}$ seems to be less than or equal to $\Syl_p(E_d/C_d)_{\rho}$. One wonders whether $\Syl_p\big( E/\mcal{S}(\mfrak{a}) \big)$ is encoding information more akin to the \emph{exponent} of $\Syl_p(\mfrak{C}_{\mfrak{a}})$.
\end{remark}

\end{document}